\documentclass[12pt]{article}
\usepackage[utf8]{inputenc}
\usepackage[T1]{fontenc}

\usepackage{authblk}  

\usepackage{amsopn}
\usepackage{lipsum}
\usepackage{amsfonts, amssymb, amsmath}
\usepackage{amsthm}
\newcommand{\email}[1]{\texttt{#1}}
\newenvironment{keywords}{\begin{center}\textbf{Keywords: }}{\end{center}}
\theoremstyle{definition}
\newtheorem{assumption}{Assumption}
\newtheorem{definition}{Definition}
\newtheorem{remark}{Remark}
\newtheorem{lemma}{Lemma}
\newtheorem{theorem}{Theorem}
\newtheorem{corollary}{Corollary}
\newtheorem{proposition}{Proposition}

\usepackage{algorithm}
\usepackage{algpseudocode}
\usepackage{epigraph}
\usepackage{enumerate}

\usepackage{pgfplots}
\pgfplotsset{compat=1.15}
\usepackage{mathrsfs}
\usetikzlibrary{arrows}

\usepackage{hyperref}

\DeclareMathOperator{\argmin}{argmin}

\renewcommand{\equiv}{:=}
\newcommand{\map}[3]{#1:\,#2\rightarrow #3\,}
\newcommand{\paren}[1]{\left(#1\right)}
\newcommand{\tbar}{{\overline{t}}}
\newcommand{\Nbb}{\mathbb{N}}

\newcommand{\R}{\mathbb{R}} 
\newcommand{\B}{\mathbb{B}}
\newcommand{\Ebb}{\mathbb{E}}

\newcommand{\dist}{\mathrm{dist}}

\newcommand{\Id}{\mathrm{Id}}

\newcommand{\gph}{\mathrm{gph} }

\newcommand{\Fix}{\mathrm{Fix}\ }

\newcommand{\conv}{\mathrm{conv}}
\newcommand{\pncone}[1]{N^{\mbox{\rm prox}}_{#1}} 

\newcommand{\Fixepsilon}{{\mathrm{Fix}_{\varepsilon} }}

\usepackage{thmtools}

\title{Cyclic Relaxed Douglas-Rachford Splitting for Inconsistent Nonconvex Feasibility}
\author{
 \begin{minipage}{\textwidth}
   \centering
   Thi Lan Dinh\thanks{Institut f\"ur Numerische und Angewandte Mathematik, Universit\"at G\"{o}ttingen, Lotzestr. 16--18, G\"ottingen 37083, Germany (\email{t.dinh@math.uni-goettingen.de}).} \quad
   G. S. Matthijs Jansen\thanks{I. Physical Institute, Universit\"at G\"ottingen, Friedrich-Hund-Platz 1,  37077 G\"{o}ttingen, Germany (\email{gsmjansen@uni-goettingen.de}).} \quad
   D. Russell Luke\thanks{Institut f\"ur Numerische und Angewandte Mathematik, Universit\"at G\"ottingen, Lotzestr. 16--18, G\"{o}ttingen, Germany (\email{r.luke@math.uni-goettingen.de}).}
 \end{minipage}
}

\date{\today}

\begin{document}

\maketitle

\begin{abstract}
We study the cyclic relaxed Douglas-Rachford algorithm for possibly nonconvex, and inconsistent feasibility problems.
This algorithm can be viewed as a convex relaxation between the cyclic Douglas-Rachford algorithm first introduced
by Borwein and Tam \cite{BorTam14} and the classical cyclic projections algorithm.  
We characterize the
fixed points of the cyclic relaxed Douglas-Rachford algorithm and show the relation of the {\em shadows} of
these fixed points to the fixed points of the cyclic projections algorithm.  Finally, we provide conditions that
guarantee local quantitative convergence estimates in the nonconvex, inconsistent setting.
\end{abstract}
\noindent



%
%
%
\section{Introduction}
Given a finite collection of closed subsets of a real Euclidean space, $\{A_1,\dots,A_m\}$, we seek
a point that is close, in some sense, to belonging to all sets simultaneously.  When the sets
have common points this is the consistent feasibility problem
\begin{equation}\label{eq:multiset.model}
    \mbox{Find } z\in \bigcap_{i=1}^m A_i.
\end{equation}
When the collection of sets does not possess common points, problem \eqref{eq:multiset.model} is
said to be {\em inconsistent}.
The leading algorithms for solving \eqref{eq:multiset.model} in both the consistent and
inconsistent cases are based on projectors onto the sets, and among these the Douglas-Rachford
algorithm is fundamental.  For inconsistent feasibility problems, the Douglas-Rachford algorithm
does not possess fixed points;  in the convex setting it even diverges \cite{LionsMercier79}.
However, where other formulations and methods get stuck in undesirable local minima, the apparent instability
of the Douglas-Rachford algorithm is a tremendous asset: it will not converge to bad local minima
for inconsistent feasibility.  Stabilizations of the Douglas-Rachford mapping have been studied
by many authors.  The form we focus on was first examined in \cite{Luke2005}.

Let $P_A$ denote the metric projector onto the closed set $A$. The {\em reflector} is defined by
$R_A:= 2P _A - \Id$, where $\Id$ is the identity mapping.
For the case of two sets, the relaxed Douglas-Rachford operator $T_{1,2}$ is defined by
\begin{equation}\label{eq:DRl12}
    T_{1,2}=\frac{\lambda}{2}(R_{A_1 }R_{A_2}+\Id)+(1-\lambda)P_{A_2},
\end{equation}
where the relaxation parameter $\lambda$ is a constant in $(0,1]$.
With $\lambda = 1$, the operator $T_{1,2}$ reduces to the original Douglas-Rachford operator.
In \cite{bauschke2004finding}, the asymptotic behavior of this
iteration is studied for the case $\lambda=1$  for inconsistent collections of closed, convex sets.
In the case where the intersection is nonempty, the authors also show weak convergence to a point whose
shadow, defined as the projection of this point onto $A_2$,
belongs to the intersection of the sets.
In the inconsistent nonconvex setting, the fixed points were characterized in \cite{Luke08}.
Under the assumption of {\em prox-regularity} of one, or both of the sets, the Douglas-Rachford algorithm
was shown to be locally linearly convergent when the feasibility problem is consistent \cite{hesse2013nonconvex,Phan16},
and similar results were obtained for a more general type of relaxation to \eqref{eq:DRl12} for consistent
feasibility in \cite{DaoPhan18}.  In \cite{luke2020convergence} algorithm \eqref{eq:DRl12} was shown to
be locally linearly convergent for inconsistent collections of sets that are a generalization of prox-regular sets
whenever the sets, as a collection, satisfy a type of {\em subtransversality} condition \cite[Theorem 4.11]{luke2020convergence}.

For the case of $m>2$, there are many different ways to generalize \eqref{eq:DRl12}.  We consider sequences
of a relaxation of the cyclic Douglas-Rachford algorithm
\begin{subequations}
\begin{equation}\label{eq:CDRl.alg.intro}
y^{(k+1)}\in T y^{(k)}\,,\ k=1,2,\dots
\end{equation}
for
\begin{equation}\label{eq:TCDRl}
    T:= T_{1,2}\circ T_{2,3} \circ \dots \circ T_{m-2,m-1} \circ T_{m-1,m} \circ T_{m,1},
\end{equation}
where 
\begin{equation}\label{eq:DRl.Ai}
    T_{i,j}=\frac{\lambda}{2}(R_{A_i}R_{A_{j}}+\Id)+(1-\lambda)P_{A_{j}}\,.
\end{equation}
\end{subequations}
We call $T$ the cyclic relaxed Douglas-Rachford operator.

If $\lambda=1$, $T$ is the Cyclic Douglas-Rachford operator first studied in a Hilbert space setting
in \cite{BorTam14}.
There the authors established weak convergence of the algorithm for the consistent problem
\eqref{eq:multiset.model} for closed convex sets;  they also characterized the fixed points in the case where the sets
have points in common.   The cyclic relaxed Douglas-Rachford operator was first studied in \cite{luke2018relaxed}
in the consistent convex setting.
If $\lambda=0$, T becomes the cyclic projection operator
$P_{A_2}\circ P_{A_3} \circ \dots \circ P_{A_{m-1}} \circ P_{A_m} \circ P_{A_1}$.
In this case, Luke, Thao and Tam prove in \cite[Theorem 3.2]{russell2018quantitative} linear convergence of the
sequence $(y^{(k)})_{k=0}^\infty$ to a fixed point of $T$ when problem \eqref{eq:multiset.model} is possibly
inconsistent under the the assumption of {\em subregularity} and {\em subtransversality} of the sets $A_j$.
In the present study, we address the case $\lambda\in (0,1]$ for nonconvex, inconsistent feasibility.

In section \ref{s:notation} we introduce the basic tools and results for our approach.  Section \ref{s:characterization}
develops the analysis of fixed points of the cyclic relaxedDouglas-Rachford mapping from various perspectives.
Theorem \ref{theo:characterization.fix.T} shows that the fixed points are convex combinations of points in the sets.
Section \ref{s:shadows} investigates the shadow sequences of the cyclic relaxed Douglas-Rachford mapping with the main result,
Theorem \ref{theo:shadow.fixed.point.subset.CP.noise}, establishing that just a single projection of the fixed points of the
cyclic relaxed Douglas-Rachford mapping gets one close to a fixed point of the cyclic projections mapping.  In section \ref{s:quant}
we develop local quantitative convergence results under established assumptions (Theorem \ref{theo:convergence.CDRl}).

\section{Notation and definitions}\label{s:notation}
Our setting is the finite-dimensional Euclidean space $\mathbb E$.  The central objects and notation used below are all found in 
\cite{rt1998wets}.  
Given set-value mapping $T:\Ebb \rightrightarrows \Ebb$, define the set of fixed points of $T$ by
$    \Fix T:=\{x\in \Ebb\,|\, x\in Tx\}.
$
The inverse operator $T^{-1}$ is defined by
$    T^{-1} (a) := \{x \in \Ebb \ | \ a \in T x \}
$
and the graph of the mapping $T$ is the set
$
    \gph (T):=\{(x,y):x\in \Ebb\,,\,y\in Tx\}
$
which is a subset of $\Ebb\times \Ebb$. 
For a subset $A\subset\Ebb$, the image of $A$ under $T$, denoted by $TA$, is defined by
\begin{equation*}
    T A:=\bigcup_{a\in A} Ta\,.
\end{equation*}
When the mapping $T:\Ebb \rightrightarrows \Ebb$ is single-valued at $y\in \Ebb$, that is if $Ty = \{z\}$ is a singleton,then
we simply write $z=Ty$.
The mapping $T$ is single-valued on $U\subset E$ if it is single-valued at all points $y\in U$.
With $T,Q:\Ebb \rightrightarrows \Ebb$ being set-value mapping, the composite $TQ$ or $T\circ Q$ from $\Ebb$ to $\Ebb$ is defined by
\begin{equation*}
    TQx= (T\circ Q)x:= \bigcup_{y\in Qx} Ty.
\end{equation*}
The distance of a point $x\in \Ebb$ to a set $A\subset \Ebb$ is denoted
$\dist(x, A)$ i.e.,
\begin{equation*}
    \dist(x, A)=\inf_{a\in A}\|x-a\|\,
\end{equation*}
and the mapping to points at which the distance is attained is the {\em projector}:
\begin{equation*}
    P_A (x) := \argmin_{a\in A} \|a - x\|.
\end{equation*}
When $A$ is {\em closed and non-empty} this set is nonempty. 
If $y \in  P_A (x)$, then $y$ is called a projection of $x$ on $A$. 
The reflector is defined by $R_ C:= 2P _C - \Id$. If $y\in R_C(x)$, then $y$ is called reflection of $x$ across $C$. We denote by $\mathbb{B}$ the open unit ball and by $\B_\delta(x)$ the open
ball with radius $\delta$ around the point $x$.

The central tool for characterizing set-regularity is the {\em normal cone}.   
The {\em proximal normal cone} of a set $\Omega$ at $\bar{a}\in \Omega$ is defined by
		 	$$ N^P_\Omega(\bar{a}) := \mbox{cone}\left(P^{-1}_\Omega\bar{a}-\bar{a}\right).$$
        	Equivalently, $\bar{a}^*\in N^P_\Omega(\bar{a})$ whenever there exists $\sigma\geq 0$ such that
			$$\langle \bar{a}^*,x-\bar{a}\rangle\leq \sigma\|x-\bar{a}\|^2 \quad(\forall x\in \Omega).$$
The limiting (proximal) normal cone of $C$ at $\bar{a}$ is defined by
			$$ N_{\Omega}(\bar{a}) := \limsup\limits_{x\to\bar{a}}N_{\Omega}^{P}(x), $$
			where the limit superior is taken in the sense of Painlev\'e--Kuratowski outer limit.
 	 When $\bar{a}\not\in \Omega$, all normal cones at $\bar{a}$ are empty (by definition).

 	 The regularity of sets is defined in terms of their (truncated) normal cones.
\begin{definition}[super-regularity at a distance]\label{definition: super-reg+}
	A set $\Omega\subset \Ebb$ is called {\em $\epsilon$-super-regular at a distance relative
	to $\Lambda\subset \Ebb$ at $\bar x$} with constant $\epsilon$ whenever	\begin{align}\label{eq:ele.subreg}
    &\exists ~U_\epsilon\subset\mathbb{E}~\mbox{ open}~ :~ \bar x\in U_\epsilon,\mbox{ and }\\
    & \langle{v - (y'-y),y - x\rangle}\leq
    \epsilon_U\|v - (y'-y)\|\|y - x\|\\
    &\forall y'\in U_\epsilon\cap \Lambda, \quad \forall y\in P_\Omega(y'),\\
    & \forall (x,v)\in\left\{(x,v)\in \gph \pncone {\Omega}~\mid~x+v\in U_\epsilon,~x\in P_\Omega(x+v)\right\}.
    \end{align}
	
	The set $\Omega$ is called {\em super-regular at a distance relative
	to $\Lambda$ at $\bar x$} if it is $\epsilon$-super-regular at a distance relative
	to $\Lambda$ at $\bar x$ for all $\epsilon>0$.
\end{definition}
Super-regular sets were first studied in \cite{lewis2009local}.
The important feature of the definition is that, unlike in \cite{lewis2009local},
the reference point $\bar x$ need not be an element
of the set $\Omega$.  This was first proposed in \cite{luke2020convergence}
and allows one to characterize the regularity of sets from the perspective
of any point, not just those in the set.

Of particular interest for the application of phase retrieval are {\em prox-regular sets} \cite[Definition 1.1]{poliquin2000local}, defined here as those sets with single-valued projector on small enough neighborhoods of the set \cite[Theorem 1.3(a) and (k)]{poliquin2000local}. Clearly any convex set is prox-regular with single-valued projector globally.
More generally, \cite[Proposition 3.1]{russell2018quantitative} yields the following implications:
\begin{equation*}
\begin{array}{ccc}
\mbox{convexity}&\Rightarrow      \mbox{prox-regularity} &\Rightarrow \epsilon\mbox{-super-regularity at a distance}.
\end{array}
\end{equation*}

We show below in Lemma \ref{lem:nonexp.refector.projector} that $\epsilon$-super-regular sets
 	 have projectors and reflectors that are {\em almost nonexpansive} \eqref{e:epsqnonexp} and
 	 {\em almost $\alpha$-firmly nonexpansive} \eqref{e:paafne}
 	 (see \cite[Definition 1]{berdellima2022alpha}).
Let $D$ be a nonempty subset of $\mathbb{E}$ and let $T$ be a (set-valued) mapping from $D$ to $\mathbb{E}$.
$T$ is said to be  {\em pointwise almost nonexpansive on $D$ at $ y \in D$} if 
\begin{align}\label{e:epsqnonexp}
\exists \epsilon\in [0,1):&\quad \|x^+- y^+\|\leq\sqrt{1+\epsilon}\|x- y\|\,,\\
&\quad \forall~ y^+\in T y \mbox{ and } \forall~ x^+\in Tx \mbox{ whenever }x\in D.
\end{align}
If \eqref{e:epsqnonexp} holds with $\epsilon= 0$ (i.e. 1-Lipschitz) at every point $y\in D$, then $T$ is  {\em nonexpansive} on D, in line with the classical definition.
$T$ is {\em pointwise almost $\alpha$-firmly nonexpensive}
(abbreviated pointwise a$\alpha$-fne) at $y\in D$ whenever
\begin{align}\label{e:paafne}
&\exists\,\alpha\in (0,1)\mbox{ and }\epsilon\in [0,1)~:\\
&\qquad \|x^+-y^+\|^2\leq (1+\epsilon)\|x-y\|^2-\frac{1-\alpha}{\alpha}\|(x^+-x)-(y^+-y)\|^2\\
    &\qquad \forall x\in D\,,\,\forall x^+\in Tx\,,\,\forall y^+\in Ty\,, \\
\end{align}
If the constant $\alpha=1/2$, then $T$ is said to be pointwise almost
firmly nonexpansive  on $D$ with violation 
$\varepsilon$ at $y$.

\begin{remark}\label{r:single-valued}
What we are calling a$\alpha$-fne mappings have been called {\em almost averaged} in earlier works (\cite{luke2020convergence,russell2018quantitative}).  The current terminology is in line with metric space theory \cite{berdellima2022alpha}, and is more convenient for the presentation.
If $T$ is pointwise a$\alpha$-fne with $\alpha=1/2$ and  violation $\epsilon=0$ for every $y\in D$, then $T$ is just firmly nonexpansive on $D$ in the classical sense.

\end{remark}

The results below have been established in \cite[Proposition 2.1]{russell2018quantitative}
\begin{lemma}
[characterizations of $\alpha$-firmly nonexpensive operators]
\label{lem:average char}
   Let $T: \Ebb \rightrightarrows \Ebb$, $U\subset\Ebb$ and let $\alpha\in (0,1)$.  The following are equivalent. 
\begin{enumerate}[(i)]
   \item\label{t:average char i} $T$ is pointwise a$\alpha$-fne at $ y$ on $U$ with violation $\epsilon$.
   \item\label{t:average char ii} $(1-\frac{1}{\alpha})\Id + \frac{1}{\alpha}T$ is pointwise almost nonexpansive at $ y$ 
	on $U\subset\Ebb$ with violation $\epsilon/\alpha$. 
\end{enumerate}
Consequently, if $T$ is pointwise a$\alpha$-fne at $y$ on $U$ with violation $\epsilon$
then $T$ is pointwise almost nonexpansive at $y$ on $U$ with
violation at most $\epsilon$.
\end{lemma}

The next lemma is the chain rule for pointwise a$\alpha$-fne mappings
(see, e.g., \cite[Proposition 2.4 (iii)]{russell2018quantitative}).
\begin{lemma}\label{t:av-comp av}
Let $T_j:\mathbb{E}\rightrightarrows \mathbb{E}$ for $j=1,2,\dots,m$
be pointwise a$\alpha$-f.n.e. on $U_j$ at all $ y_{j}\in S_j\subset \mathbb{E}$
with violation $\epsilon_{U_j}$ and
constant $\alpha_{U_j}\in (0, 1)$ where
$U_j\supseteq S_j$ for $j=1,2,\dots,m$.
If 
$T_{j}U_{j}\subseteq U_{j-1}$ and $T_{j}S_{j}\subseteq S_{j-1}$ for $j=2,3,\dots,m$, 
then the composite mapping $T:=T_1\circ T_{2}\circ \cdots \circ  T_m$ is pointwise a$\alpha$-fne at all 
$y\in S_m$ on $U_m$ with violation at most
\begin{subequations}
\begin{equation}
    \epsilon_T=\prod_{j=1}^m(1+\epsilon_{U_j})-1
\end{equation}
and averaging constant at least
\begin{equation}
 \alpha_T = \frac{m}{m-1 + \frac{1}{\max_{j=1,2,\dots,m}{\alpha_{U_j}}}}.
\end{equation}%
\end{subequations}
\end{lemma}

We now recall the characterization of set regularities and the inheritance of these characterizations on the
corresponding projectors.  The lemma below recalls sufficient conditions for projectors and reflectors to be a$\alpha$-fne.
\begin{lemma}[projectors/reflectors of $\epsilon$-super-regular sets, Proposition 3.4, \cite{luke2020convergence}]
\label{lem:nonexp.refector.projector}
Let $\Omega\subset\Ebb$ be nonempty closed, and let $U$ be a neighborhood of $\bar x\in \Omega$.
Let $\Lambda:= P_{\Omega}^{-1}(\bar x)\cap U$.
If $\Omega$ is $\epsilon$-super-regular at a distance at $\bar x$
relative to $\Lambda$ with constant $\epsilon_U$ on the neighborhood $U$ \eqref{eq:ele.subreg}, then the following statements hold:
\begin{enumerate}[(i)]
    \item If $\epsilon_U\in [0,1)$, the projector $P_\Omega$ satisfies
    \begin{align*}
    &\forall y\in\Lambda, \quad \|x^+-y^+\|\leq\sqrt{1+\hat{\epsilon}} \|x-y\|
\quad\forall x\in U,\, x^+\in P_\Omega x,\, y^+\in P_\Omega y\\
    &\mbox{ with } \hat{\epsilon}= 4\epsilon_U/(1-\epsilon_U)^2
\end{align*}
and
    \begin{align*}
    &\forall y\in\Lambda, \quad \|x^+-y^+\|^2 +  \|(x^+-x)-(y^+-y)\|^2
    \leq (1+\check{\epsilon}) \|x-y\|^2\\
&\,\forall x\in U,\, x^+\in P_\Omega x,\, y^+\in P_\Omega y\\
    &\mbox{ with } \check{\epsilon}= 4\epsilon_U(1+\epsilon_U)/(1-\epsilon_U)^2
\end{align*}
In other words, for $\epsilon_U\in [0,2\sqrt{3}/3-1)$\footnote{The restriction on $\epsilon_U$ is
required by definition for $P_\Omega$ to be a$\alpha$fne.} the projector $P_\Omega$ is
pointwise almost nonexpansive on $U$ at each $y\in \Lambda$
with violation $\hat\epsilon$,
and pointwise a$\alpha$-fne at each $y\in \Lambda$ with constant $\alpha=1/2$ and violation $\check\epsilon$.
\item If $\epsilon_U\in [0,1)$, the reflector $R_\Omega$ satisfies
\begin{align}\label{e:epstilde}
&\forall y\in\Lambda, \quad    \|x^+-y^+\|\leq\sqrt{1+\tilde{\epsilon}} \|x-y\|
\quad\forall x\in U,\, x^+\in R_\Omega x,\, y^+\in R_\Omega y\\
&\mbox{ with } \tilde{\epsilon} = 8\epsilon_U(1+\epsilon_U)/(1-\epsilon_U)^2
\end{align}

In other words, for $\epsilon_U\in [0,4\sqrt{2}/7-5/7)$\footnote{The restriction on $\epsilon_U$ is
required by definition for $R_\Omega$ to be a$\alpha$fne.} the reflector
$R_\Omega$
is pointwise almost nonexpansive at each $y\in\Lambda$
with violation $\tilde{\epsilon}$.
\end{enumerate}
\end{lemma}

In order to quantify the error in approximating the true projector with its linearization,
we require one final notion of set regularity.  The tangent cone (contingent cone) to a
closed subset $A\subset \Ebb$ at $\bar x\in A$ is defined by \cite[Chapter 4]{Aubin2009}
\begin{equation}
 \mathcal{T}_{A}(\bar x):=\left\{v\in \Ebb\,\mid\,\liminf _{h\to 0^{+}}{\frac {d_{A}(\bar x+hv)}{h}}=0\right\}.
\end{equation}
We say that a subset $A\subset \Ebb$ has the {\em Shapiro$^+$ property} at $\bar x\in A$ on a neighborhood
$V$ of $\bar x$ with constant $k>0$ if
\begin{equation}\label{e:Shapiro+}
    d_{\mathcal{T}_{A}(y)}(x-y)\leq k\|x-y\|,\ \forall x\in V\,,\,\forall y\in A\cap V.
\end{equation}
This definition is a specialization of the of the property stated in \cite{shapiro1994existence} to the case $m=1$.
Our definition is an extension in the sense that we do not require $x\in A$.


\section{Analysis of Fixed Points}
\label{s:characterization}
In this section, we  characterize the fixed points of a cyclic relaxed Douglas-Rachford operator as a subset of the convex hull of  points in the sets associated with the operator.  Since the interpretation of the fixed points is difficult, we show that the {\em shadows} of these points are easily interpretable in terms of fixed points of the cyclic projections mapping. 

\subsection{Characterization of fixed points}\label{sec:characterization.fixed.points}
The following lemma gives another representation of relaxed Douglas-Rachford operators. 
\begin{lemma}\label{lem:Tij.expression}
Let $A_i,A_j$ be closed subsets of $\Ebb$ and $\lambda\in [0,1]$. Let $T_{j,j+1}$ be the relaxed Douglas-Rachford operator defined as in \eqref{eq:DRl.Ai}. Let $x\in\Ebb$. Assume that $\{y\}=T_{i,j}x$. Then
\begin{equation}
    y = \lambda P_{A_j}R_{A_{j+1}}x+(1-2\lambda)P_{A_{j+1}}x+\lambda x.
\end{equation}
\end{lemma}
\begin{proof}
    By definition of relaxed Douglas-Rachford operators, we have 
\begin{equation}
    \begin{array}{rl}
        y=T_{j,j+1}x =&\frac{\lambda}{2}(R_{A_j}R_{A_{j+1}}x+x)+(1-\lambda)P_{A_{j+1}}x  \\
         =& \frac{\lambda}{2}(2P_{A_j}R_{A_{j+1}}x-R_{A_{j+1}}x+x)+(1-\lambda)P_{A_{j+1}}x\\
         =& \frac{\lambda}{2}(2P_{A_j}R_{A_{j+1}}x-2P_{A_{j+1}}x+2x)+(1-\lambda)P_{A_{j+1}}x\\
         =&\lambda P_{A_j}R_{A_{j+1}}x + (1-2\lambda)P_{A_{j+1}}x+\lambda x.
    \end{array}
\end{equation}
 This concludes the proof.
\end{proof}

The next theorem says that each fixed point of a cyclic relaxed Douglas-Rachford operator is a linear combination of $2m$ points in $m$ sets associated with the operator; for
$\lambda\in[0,\frac{1}{2}]$, the linear combination is a convex combination.
\begin{theorem}[fixed points]\label{theo:characterization.fix.T}
Let $A_j$ be nonempty closed subset of $\mathbb E$, $j=1,\dots,m$.
Denote $A_{m+1}:=A_1$.
Let $\lambda\in [0,1]$.
Define $T\equiv T_{1,2}\circ T_{2,3}\circ\dots\circ T_{m-1,m}\circ T_{m,m+1}$, where $T_{j,j+1}$ defined as in \eqref{eq:DRl.Ai}.
Assume that $\Fix T\neq\emptyset$ and 
$T_{i,i+1}$ is single-valued on $T_{i+1,i+2}\circ\dots\circ T_{m,m+1}\Fix T$, $i=1,\dots,m$.
Set
\begin{equation}
    \mathcal{M}_1:=\left\{
    x~\middle|~ 
    \begin{array}{rl}
    &x=\sum_{j=1}^{m}[\frac{\lambda^j}{1-\lambda^m}P_{A_j}y_{j+1}+\frac{(1-2\lambda)\lambda^{j-1}}{1-\lambda^m}P_{A_{j+1}}x_{j+1}],\\
    &x_{m+1}=x,\\
    &x_j=T_{j,j+1}x_{j+1},\qquad j=m,\dots,1,\\
    
    &y_{j+1}\in R_{A_{j+1} }x_{j+1},\qquad j=1,\dots,m
    \end{array}
    \right\},
\end{equation}
and
\begin{equation}
    \mathcal{M}_2:=\left\{
    x~\middle|~ 
    \begin{array}{rl}
    &P_{A_1}x=\sum_{j=1}^{m}P_{A_j}y_{j+1}-\sum_{j=1}^{m-1}P_{A_{j+1}}x_{j+1},\\
    &x_{m+1}=x,\\
    &x_j=T_{j,j+1}x_{j+1},\qquad j=m,\dots,1,\\
    
    &y_{j+1}\in R_{A_{j+1} }x_{j+1},\qquad j=1,\dots,m
    \end{array}
    \right\}.
\end{equation}
\begin{enumerate}[(i)]
    \item\label{theo:characterization.fix.T i} If $\lambda\in [0,1)$, then $\sum_{j=1}^m [\frac{\lambda^j}{1-\lambda^m}+\frac{(1-2\lambda)\lambda^{j-1}}{1-\lambda^m}]=1$ and $\Fix T= \mathcal{M}_1$.
    \item\label{theo:characterization.fix.T ii}
If $\lambda\in[0,\frac{1}{2}]$, then $\frac{\lambda^j}{1-\lambda^m}$ and $\frac{(1-2\lambda)\lambda^{j-1}}{1-\lambda^m}$ are nonnegative.
\item\label{theo:characterization.fix.T iii} If $\lambda=1$, then  $\Fix T = \mathcal{M}_2$
\end{enumerate}
\end{theorem}
\begin{proof}
Part \eqref{theo:characterization.fix.T i}.
By assumption, $T$ is single-valued on $\Fix T$,
so $z=Tz$ for all $z\in \Fix T$. Let $x$ be any point and set $x_{m+1}:=x$ and
$x_j:=T_{j,j+1}x_{j+1}$, $j=m,\dots,1$.  Let $y_{j+1}\in R_{A_{j+1} }x_{j+1}$, $j=1,\dots,m$.
We first show that $x\in\Fix T$ is equivalent to $x_1=x$. Indeed, we have
\begin{equation}\label{eq:x1.equal.x}
    x_1=T_{1,2}x_2=T_{1,2}\circ T_{2,3}x_3=T_{1,2}\circ T_{2,3}\circ\dots\circ T_{m,m+1}x_{m+1}=Tx=x,
\end{equation}
hence $x\in\Fix T$ if and only if $x_1=x$.
Next we show that $x_1=x$ if and only if $x\in \mathcal{M}_1$.  To see this, note that \eqref{eq:x1.equal.x} and Lemma \ref{lem:Tij.expression} yield
\begin{equation}\label{e:inter}
    \begin{array}{rl}
         x=x_1=T_{1,2}x_2
         =& \lambda P_{A_1}y_2+(1-2\lambda)P_{A_2}x_2+\lambda x_2\\
         =& \lambda P_{A_1}y_2+(1-2\lambda)P_{A_2}x_2+\lambda T_{2,3}x_3\\
         =&\lambda P_{A_1}y_2 + (1-2\lambda)P_{A_2}x_2 +\lambda[\lambda P_{A_2}y_3+(1-2\lambda)P_{A_3}x_3+\lambda x_3]\\
         =&\lambda P_{A_1}y_2+ (1-2\lambda)P_{A_2}x_2+\lambda^2 P_{A_2}y_3+\lambda(1-2\lambda)P_{A_3}x_3+\lambda^2 x_3\\
         =&\dots\\
         =&
         \sum_{j=1}^m \lambda^jP_{A_j}y_{j+1}+(1-2\lambda)\sum_{j=1}^m \lambda^{j-1}P_{A_{j+1}}x_{j+1}+\lambda^mx_{m+1}.
    \end{array}
\end{equation}
Since $x_{m+1}=x$, \eqref{e:inter} is equivalent to
\begin{equation}\label{eq:.x}
    (1-\lambda^m)x=\sum_{j=1}^m \lambda^jP_{A_j}y_{j+1}+(1-2\lambda)\sum_{j=1}^m \lambda^{j-1}P_{A_{j+1}}x_{j+1},
\end{equation}
which is equivalent to
\begin{equation}
    x=\sum_{j=1}^m\left[ \frac{\lambda^j}{1-\lambda^m}P_{A_j}y_{j+1}+  \frac{(1-2\lambda)\lambda^{j-1}}{1-\lambda^m}P_{A_{j+1}}x_{j+1}\right].
\end{equation}
In other words, $x\in\mathcal{M}_1$.  Since all of the above are equivalence relations, this yields the claimed equivalence of $\Fix T$ and $\mathcal{M}_1$.

In addition, we have
\begin{equation}
\begin{array}{rl}
    \sum_{j=1}^m \left[\frac{\lambda^j}{1-\lambda^m}+\frac{(1-2\lambda)\lambda^{j-1}}{1-\lambda^m}\right]=&\frac{1}{1-\lambda^m}\sum_{j=1}^m \lambda^j+\frac{1-2\lambda}{1-\lambda^m}\sum_{j=1}^m \lambda^{j-1}\\[5pt]
    =&\frac{1}{1-\lambda^m}\frac{\lambda(1-\lambda^m)}{1-\lambda}+\frac{1-2\lambda}{1-\lambda^m}\frac{1-\lambda^m}{1-\lambda}\\[5pt]
    =&\frac{\lambda}{1-\lambda}+\frac{1-2\lambda}{1-\lambda}=1.
\end{array}
\end{equation}

Part \eqref{theo:characterization.fix.T ii}. This is clear.

Part \eqref{theo:characterization.fix.T iii}. By \eqref{eq:.x}, with $\lambda=1$, the result is immediate.
\end{proof}
\begin{remark}
In \cite[Theorem 3.13]{luke2020convergence}, Luke and Martins demonstrated that each fixed point of a relaxed Douglas-Rachford operator lies on a line segment, with one endpoint belonging to the ``inner'' set (i.e. the first reflected set in the Douglas-Rachford operator). Our characterization of the fixed points of the cyclic relaxed Douglas-Rachford operator in Theorem \ref{theo:characterization.fix.T} shares a similar structure but one important difference.
In \cite[Theorem 3.13]{luke2020convergence} the fixed points diverge to the horizon in the direction of the gap as the relaxation parameter $\lambda$ goes to $1$.  In our case the fixed points lie in the convex hull of $2m$ points distributed within the $m$ sets, each set containing two points.  Thus, for all  $\lambda< 1$ the fixed points are bounded if the sets are bounded.  In other words, if the sets are bounded, the cyclic relaxed Douglas-Rachford algorithm cannot diverge, and hence possesses cluster points.  This coincides with what was known in the convex-consistent case \cite[Theorem 3.1]{BorTam14} and \cite[Theorem 2]{luke2018relaxed}, but does not rely upon convexity or consistency.
\end{remark}
Denote by $\conv(S)$ the convex hull of a given set $S\subset \mathbb E$.
The following corollary is a direct consequence of Theorem \ref{theo:characterization.fix.T}.

\begin{corollary}\label{coro:fixed.points.in.convexhull}
Let all assumptions and notation of Theorem \ref{theo:characterization.fix.T} hold.
Assume $\lambda\in[0,\frac{1}{2}]$.
Then
\begin{equation}
    \Fix T\subseteq \conv \left(\bigcup_{j=1}^m A_j\right).
\end{equation}
\end{corollary}
\begin{remark}\label{rem:intersection.subset.fixed.points}
Let $T$ be the cyclic relaxed Douglas-Rachford operator with $\lambda\in[0,1]$.
If problem \eqref{eq:multiset.model} is consistent, i.e., $\bigcap_{i=1}^m A_i\ne \emptyset$, then $\bigcap_{i=1}^m A_i \subseteq \Fix T$.
Indeed, take $y\in\bigcap_{i=1}^m A_i$, then  $y\in A_i$, for $i=1,\dots,m$. 
This gives $P_{A_i}y= y$, for $i=1,\dots,m$. We have 
\begin{equation*}
\begin{array}{rl}
    T_{i,i+1}y =&\frac{\lambda}{2}(R_{A_i}R_{A_{i+1}}+\Id)y+(1-\lambda)P_{A_{i+1}}y\\
    =&\frac{\lambda}{2}(R_{A_i}(2P_{A_{i+1}}y-y)+y)+(1-\lambda)y\\
    =&\frac{\lambda}{2}(R_{A_i}(2y-y)+y)+(1-\lambda)y\\
    =&\frac{\lambda}{2}((2P_{A_i}y-y)+y) + (1-\lambda)y\\
    =&\frac{\lambda}{2}((2y-y)+y) + (1-\lambda)y\\
    =&y.
    \end{array}
\end{equation*}
Therefore, we obtain
\begin{equation*}
\begin{array}{rl}
    Ty =& T_{1,2}\circ T_{2,3}\circ\dots \circ T_{m-1,m}\circ T_{m,m+1}y\\
    =&T_{1,2}\circ T_{2,3}\circ\dots \circ T_{m-1,m}y\\
    &\dots\\
    =&T_{1,2}y = y.
\end{array}
\end{equation*}
In the case where the sets $A_i$ are convex, it is shown in \cite[Proposition 1]{luke2018relaxed} $\lambda\in (0,1)$ that $\bigcap_{i=1}^m A_i = \Fix T$.
These results indicate that the algorithm can find points in the intersection of the sets in consistent case.
Theorem  \ref{theo:characterization.fix.T}, extends this to inconsistent and nonconvex feasibility.
\end{remark}
\begin{remark}
    Remark \ref{rem:intersection.subset.fixed.points} indicates that  $\Fix T$ contains the intersection of the sets $A_1,\dots,A_m$. 
    Corollary \ref{coro:fixed.points.in.convexhull} provides further information saying that $\Fix T$ is a subset of the convex hull of the union of the sets $A_1,\dots,A_m$.
    In short, we obtain
    \begin{equation}
        \bigcap_{i=1}^m A_i\quad
        \overset{{\lambda\in[0,1]}}{\subseteq}
        \quad\Fix T\quad
        \overset{{\lambda\in[0,\tfrac12]}}{\subseteq}
        \quad \conv \left(\bigcup_{i=1}^m A_i\right).
    \end{equation}
\if{    These inclusions are illustrated in Figure \ref{fig:fixed.points.in.convex.hull.union}
    \begin{figure}
        \centering
        \definecolor{zzttqq}{rgb}{0.6,0.2,0}
\begin{tikzpicture}[line cap=round,line join=round,>=triangle 45,x=1cm,y=1cm]
\clip(-4.34,-3.24) rectangle (3.42,2.64);
\fill[line width=1pt,color=zzttqq,fill=zzttqq,fill opacity=0.10000000149011612] (-1.36,0.28) -- (-1.42,-1.2) -- (0.66,-1.36) -- (0.64,0.12) -- cycle;
\draw [line width=1pt,dash pattern=on 1pt off 1pt] (-1.94,2.42)-- (1.26,2.16);
\draw [line width=1pt] (-2.1044006364236094,0.3966075517094181) circle (2.0300602377899635cm);
\draw [line width=1pt] (1.0945500407466338,0.12369280918933612) circle (2.0430175389272014cm);
\draw [line width=1pt] (-0.49125363748246487,-1.586468573927244) circle (1.3948664704085714cm);
\draw [line width=1pt,dash pattern=on 1pt off 1pt] (-3.311517430263942,-1.2355730801706688)-- (-1.3206707701622409,-2.707949599307221);
\draw [line width=1pt,dash pattern=on 1pt off 1pt] (0.32565026558487375,-2.7170969317103904)-- (2.102287290052563,-1.6534921279304555);
\draw (-2.6,1.42) node[anchor=north west] {$A_1$};
\draw (1.38,1.14) node[anchor=north west] {$A_2$};
\draw (-0.68,-1.8) node[anchor=north west] {$A_3$};
\draw [line width=1pt,color=zzttqq] (-1.36,0.28)-- (-1.42,-1.2);
\draw [line width=1pt,color=zzttqq] (-1.42,-1.2)-- (0.66,-1.36);
\draw [line width=1pt,color=zzttqq] (0.66,-1.36)-- (0.64,0.12);
\draw [line width=1pt,color=zzttqq] (0.64,0.12)-- (-1.36,0.28);
\draw (-0.6,-0.7) node[anchor=north west] {$\mathrm{Fix}\ T$};
\end{tikzpicture}
        \caption{Illustration of the fixed points}
        \label{fig:fixed.points.in.convex.hull.union}
    \end{figure}
    }\fi
\end{remark}

The following proposition provides a comparison of the shadow of the fixed points of the cyclic relaxed Douglas-Rachford operator with the fixed points of cyclic projections for affine constraints.

\begin{proposition}\label{prop:shadow.fixed.point.subset.CP}
    Let $A_j$ be nonempty closed affine subsets of $\mathbb E$, $j=1,\dots,m$.
Denote $A_{m+1}:=A_1$.
Let $\lambda\in[0,1]$.
Define $T_{CDR\lambda}\equiv T_{1,2}\circ T_{2,3}\circ\dots\circ T_{m-1,m}\circ T_{m,m+1}$, where $T_{j,j+1}$
defined as in \eqref{eq:DRl.Ai}.
Assume that $\Fix T_{CDR\lambda}\neq\emptyset$ and that $T_{i,i+1}$ is single-valued on $T_{i+1,i+2}\circ\dots\circ T_{m,m+1}\Fix T_{CDR\lambda}$, $i=1,\dots,m$.
Define $T_{\overline{CP}}\equiv P_{A_1}\circ P_{A_2}\circ\dots\circ P_{A_{m}}\circ P_{A_1}$. Then
\begin{enumerate}[(i)]
\item\label{prop:shadow.fixed.point.subset.CP i}
\begin{equation}\label{eq:shadow.sunset.CP}
    P_{A_1}(\Fix T_{CDR\lambda})\subseteq \Fix T_{\overline{CP}},
\end{equation}
consequently, if $\Fix T_{CDR\lambda}$ is nonempty, then $\Fix T_{\overline{CP}}$ is also nonempty.
\item\label{prop:shadow.fixed.point.subset.CP ii}
\begin{equation}\label{e:extended CP}
    T_{\overline{CP}}(\Fix T_{CDR\lambda})\subseteq \Fix T_{\overline{CP}}.
\end{equation}
\item\label{prop:shadow.fixed.point.subset.CP iii}
\begin{equation}\label{e:CDR0}
    P_{A_1}T_{CDR0}(\Fix T_{CDR\lambda})\subseteq \Fix P_{A_1}T_{CDR0}.
\end{equation}
\end{enumerate}
\end{proposition}
\begin{proof}
Part \eqref{prop:shadow.fixed.point.subset.CP i}.  Let $x\in\Fix T_{CDR\lambda}$.
Set $x_{m+1}=x$. Set $x_i=T_{i,i+1}x_{i+1}$, for all $i=m,\dots,1$. Using \eqref{eq:x1.equal.x} we get $x_1=x$.
By Lemma \ref{lem:Tij.expression}, we have
\begin{equation}\label{eq:rep.x}
    \begin{array}{rl}
         x=x_1=T_{1,2}x_2 =&\lambda P_{A_1}R_{A_2}x_2+(1-2\lambda)P_{A_2}x_2+\lambda x_2 \\
         =& \lambda P_{A_1}R_{A_2}x_2-\lambda(2P_{A_2}x_2-x_2)+P_{A_2}x_2\\
         =&\lambda P_{A_1}R_{A_2}x_2-\lambda R_{A_2}x_2+P_{A_2}x_2.
    \end{array}
\end{equation}
From this and the assumption that $A_1$ is affine, it follows that
\begin{equation}\label{eq:PA1x.equal.PA1PA2x}
    P_{A_1}x=P_{A_1}x_1=\lambda P_{A_1}R_{A_2}x_2-\lambda P_{A_1} R_{A_2}x_2+P_{A_1}P_{A_2}x_2=P_{A_1}P_{A_2}x_2.
\end{equation}
Repeating this for $P_{A_i}$, we obtain
\begin{equation}
    P_{A_i}x_i = P_{A_i}P_{A_{i+1}}x_{i+1}, \quad  i=1,\dots,m.
\end{equation}
This implies that
\begin{equation}
    \begin{array}{rl}
         P_{A_1}x=P_{A_1}x_1=& P_{A_1}P_{A_2}x_2\\ =&P_{A_1}P_{A_2}P_{A_3}x_3\\
         =& \dots\\
        =&P_{A_1}P_{A_2}P_{A_3}\dots P_{A_m}P_{A_{m+1}}x_{m+1}\\
        =&T_{\overline{CP}}x = T_{\overline{CP}}P_{A_1}x.
    \end{array}
\end{equation}
Thus $P_{A_1}x\in\Fix  T_{\overline{CP}}$, as claimed.

Parts \eqref{prop:shadow.fixed.point.subset.CP ii} and \eqref{prop:shadow.fixed.point.subset.CP iii}
follow easily from \eqref{prop:shadow.fixed.point.subset.CP i}.
%
\end{proof}

\begin{remark}\label{rem:fix.CP.bar.eq.fix.CP}
The use of the extended cyclic projections mapping $T_{\overline{CP}}$ defined in
Proposition \ref{prop:shadow.fixed.point.subset.CP} is purely technical.
It is easy to see that  $\Fix T_{\overline{CP}}= \Fix (P_{A_1}\dots P_{A_m})$.
\end{remark}

In practice, have we monitored the shadow sequence $|P_{A_1} x^{(j)} - P_{A_1} x^{(j-1)}|$ and observed that for some initial points $x^{(0)}$, the sequence appears to converge to an {\em orbit} with $|x^{(j)} - x^{(j-1)}| \to \mbox{constant}$,
but the shadows converge:  $|P_{A_1} x^{(j)} - P_{A_1} x^{(j-1)}| \to 0$ as $j \to \infty$. This does not contradict our analysis.  In the
next section we take a closer look at the shadow sequence.

\subsection{Analysis of the shadows - affine approximations}\label{s:shadows}
Proposition \ref{prop:shadow.fixed.point.subset.CP} shows that in the affine case, the shadows of the CDR$\lambda$ fixed points on the set $A_1$ (i.e. the projections) exactly coincide with fixed points of the cyclic projection mapping.  The fixed points of the cyclic projections mapping are more easily interpreted as local gaps between successive sets in the limiting projection cycle \eqref{eq:sum.gap}.  In the non-affine case this will not hold, but it provides a heuristic for the common practice in the literature of displaying not the final iterate of the Douglas-Rachford-type algorithm, but rather the projection of this point onto one of the sets to ``clean up'' the final iterate.  In this section we put this heuristic on solid theoretical grounds.

The goal of this section is to show that the projection mappings upon which the cyclic relaxed Douglas-Rachford
mapping is built, can be approximated locally by the projection onto the tangent spaces to the sets.  This
approximation requires the Shapiro property \eqref{e:Shapiro+}.  To link quantitatively
the regularity of sets to the corresponding regularity of this projectors
requires the calculus of a$\alpha$-fne mappings as developed in \cite{russell2018quantitative}.

To begin, let $\mathbb S$ (resp. $\mathbb B$) denote the closed unit sphere (resp. ball).
Denote by $B(a,r)$ the open ball centered at $a\in\mathbb E$ with radius $r>0$.
We say that a set-valued mapping $T:\mathbb E\rightrightarrows \mathbb E$ is approximated pointwise by a set-value mapping $\Phi:\mathbb E\rightrightarrows \mathbb E$ on a subset $U\subset \mathbb E$ with error $\varepsilon>0$ if 
\begin{equation}
\arraycolsep=1.4pt\def\arraystretch{.5}
\sup_{\begin{array}{cc}
\scriptstyle x\in U\\
\scriptstyle y\in Tx\\
\scriptstyle z\in \Phi x
\end{array}} \|y-z\|\le \varepsilon.
\end{equation}
Equivalently, for every $x\in U$, for all $y\in Tx$, for all $z\in \Phi x$, $y-z=\varepsilon v$ for some $v\in \mathbb B$.

The following lemma states that the composition of two operators that are approximated by affine operators is an operator that is approximated by an affine operator.
\begin{lemma}\label{lem:composition.affine.2}
Let $T_i:\mathbb E\rightrightarrows \mathbb E$ be an operator that is approximated by an affine operator $\Phi_i:\mathbb E\to \mathbb E$ on a set $U_i\subset \mathbb E$ with error $\varepsilon_i>0$,  $i=1,2$.
Assume that $T_2U_2\subseteq U_1$.
Then $T_1\circ T_2$ is approximated by the affine operator $\Phi_{1,2}$ on $U_2$ with
\begin{equation}
\Phi_{1,2}=(1-\varepsilon_2)\Phi_1(\frac{1}{1-\varepsilon_2}\Phi_2)
\end{equation}
and error
\begin{equation}
\left\|T_1\circ T_2x- \Phi_{1,2}x \right\|\leq
\varepsilon_{1,2}:=\varepsilon_1+\varepsilon_2 \sup_{w\in\mathbb B}{\|\Phi_1w\|}\quad \forall x\in U_2.
\end{equation}
\end{lemma}
\begin{proof}
Let $x\in U_2$.
Let $y\in T_2x$.
By assumption, $y=\Phi_2 x +\varepsilon_2 v_2$ for some $v_2\in\mathbb B$.
Since $T_2U_2\subseteq U_1$, we get $y\in U_1$.
Let $z\in T_1y$.
By assumption, there exists $v_1\in \mathbb B$ such that
\begin{equation}
\begin{array}{rl}
z=&\Phi_1 y+\varepsilon_1 v_1\\
=&\Phi_1 (\Phi_2 x +\varepsilon_2 v_2)+\varepsilon_1 v_1\\
=&\Phi_1 ((1-\varepsilon_2)\frac{1}{1-\varepsilon_2}\Phi_2 x +\varepsilon_2 v_2)+\varepsilon_1 v_1\\
=&(1-\varepsilon_2)\Phi_1 (\frac{1}{1-\varepsilon_2}\Phi_2 x) +\varepsilon_2 \Phi_1v_2+\varepsilon_1 v_1\\
=&\Phi_{1,2} x +\varepsilon_{1,2} v\,,
\end{array}
\end{equation}
where $v=\frac{1}{\varepsilon_{1,2}}(\varepsilon_2 \Phi_1v_2+\varepsilon_1 v_1)$.
Note that $\|v\|\le 1$.
Hence the result follows.
\end{proof}

The following lemma extends the result of Lemma \ref{lem:composition.affine.2} to the case of more than two operators.
\begin{lemma}\label{lem:composition.affine.m}
Let $T_i:\mathbb E\rightrightarrows \mathbb E$ be operator that is approximated by an affine operator $\Phi_i:\mathbb E\to \mathbb E$ and error $\varepsilon_i>0$ on a set $U_i\subset \mathbb E$, $i=1,2,\dots,m$.
Assume that $T_{i+1}U_{i+1}\subseteq U_i$, $i=1,\dots,m-1$.
For $i=1,\dots,m-1$, define
\begin{subequations}
\begin{equation}\label{eq:phi.1.to.i}
\Phi_{1,\dots,i+1}\equiv(1-\varepsilon_{i+1})\Phi_{1,\dots,i}(\frac{1}{1-\varepsilon_{i+1}}\Phi_{i+1}\cdot)\,,
\end{equation}
with
\begin{equation}\label{eq:esp.1.to.i}
\varepsilon_{1,\dots,i+1}\equiv\varepsilon_{1,\dots,i}+\varepsilon_{i+1} \sup_{w\in\mathbb B}{\|\Phi_{1,\dots,i}w\|}\,.
\end{equation}
\end{subequations}
Then $T_1\circ T_2 \circ \dots \circ T_m$ is approximated by affine operator $\Phi_{1,\dots,m}$ with error $\varepsilon_{1,\dots,m}$ on $U_m$.
\end{lemma}
\begin{proof}
    We prove the result by induction for $m=2,3,\dots$.
    The result with $m=2$  follows directly from Lemma \ref{lem:composition.affine.2}.
    Assume that the result holds with $m=i$, i.e., $T_1\circ \dots\circ T_i$ is approximated by affine operator $\Phi_{1,\dots,i}$ with error $\varepsilon_{1,\dots,i}$ on $U_i$.
    We will show that the result holds with $m=i+1$.
    Using Lemma \ref{lem:composition.affine.2} with $T_1\leftarrow T_1\circ \dots\circ T_i$, $\Phi_1\leftarrow \Phi_{1,\dots,i}$, $\varepsilon_1\leftarrow \varepsilon_{1,\dots,i}$, $T_2\leftarrow T_{i+1}$, $\Phi_2\leftarrow \Phi_{i+1}$, $\varepsilon_2\leftarrow \varepsilon_{i+1}$, we get $T_1\circ \dots\circ T_{i+1}=(T_1\circ \dots\circ T_i)\circ T_{i+1}$ is approximated by affine approximation $\Phi_{1,\dots,i+1}$ and error $\varepsilon_{i+1}$ on $U_{i+1}$, defined as in \eqref{eq:phi.1.to.i} and \eqref{eq:esp.1.to.i}, respectively. 
    Hence the proof is completed.
\end{proof}

Given the operator $T:\mathbb E\rightrightarrows \mathbb E$ and constant $\varepsilon>0$, denote
\begin{equation}
\Fixepsilon T\equiv\{y\in\mathbb E\,:\, y\in Ty +\varepsilon \mathbb B\}.
\end{equation}
We call $\Fixepsilon T$ the set of {\em almost} fixed points of $T$ with error $\varepsilon$.
If $\varepsilon_1>\varepsilon_2>0$, then $\Fix T=\mathrm{Fix}_{0}T\subset \mathrm{Fix}_{\varepsilon_2}T\subset\mathrm{Fix}_{\varepsilon_1}T$.

The following lemma states that if projections onto the sets in the feasibility model is approximated by affine operators in certain neighborhoods, the shadows of fixed points of the cyclic relaxed Douglas-Rachford operator are almost fixed points of the cyclic projection operator.
\begin{lemma}\label{lem:shadow.fixed.point.subset.CP.noise}
Let $A_j$ be nonempty closed subsets of $\mathbb E$, $j=1,\dots,m$.
Denote $A_{m+1}:=A_1$.
Let $\lambda\in[0,1]$.
Set $T_{i,\dots,m}\equiv T_{i,i+1}\circ T_{i+1,i+2}\circ\dots\circ T_{m,m+1}$.
Let $U$ be a subset of $\mathbb E$. 
Assume that the following conditions hold:
\begin{enumerate}[(a)]
    \item $\Fix T_{CDR\lambda}\cap U\neq\emptyset$ and  $T_{i,i+1}$ is single-valued on $T_{i+1,\dots,m}(\Fix T_{CDR\lambda}\cap U)$, $i=1,\dots,m$.
    \item $P_{A_i}$ is approximated by an affine operator $\Phi_i:\mathbb E\to \mathbb E$ with error $\varepsilon_i>0$ on a set $U_i\subset \mathbb E$, $i=1,2,\dots,m$.
    \item For $i=1,\dots,m-1$, $P_{A_{i+1}}U_{i+1}\subseteq U_i$, and for $i=1,\dots,m$,
\begin{equation}\label{eq:subset.Ui}
\begin{array}{rl}
&T_{i,\dots,m}(\Fix T_{CDR\lambda}\cap U)\subseteq U_i,\\
&P_{A_i}R_{A_{i+1}}T_{i+1,\dots,m}(\Fix T_{CDR\lambda}\cap U)\subseteq U_i,\\
&R_{A_{i+1}}T_{i+1,\dots,m}(\Fix T_{CDR\lambda}\cap U)\subseteq U_i,\\
&P_{A_{i+1}}T_{i+1,\dots,m}(\Fix T_{CDR\lambda}\cap U)\subseteq U_i\,.
\end{array}
\end{equation}
\end{enumerate}
Define
\begin{subequations}
\begin{equation}
(i=1,\dots,m-1),\quad \Phi_{1,\dots,i+1}\equiv(1-\varepsilon_{i+1})\Phi_{1,\dots,i}(\frac{1}{1-\varepsilon_{i+1}}\Phi_{i+1})\,,
\end{equation}
\begin{equation}
\varepsilon_{1,\dots,i+1}\equiv\varepsilon_{1,\dots,i}+\varepsilon_{i+1} \sup_{w\in\mathbb B}{\|\Phi_{1,\dots,i}w\|}\,.
\end{equation}
and
\end{subequations}
$T_{\overline{CP}}\equiv P_{A_1}\circ P_{A_2}\circ\dots\circ P_{A_{m}}\circ P_{A_1}$.

Then
\begin{equation}\label{eq:shadow.sunset.CP.noise}
    P_{A_1}(\Fix T_{CDR\lambda}\cap U)\subseteq \Fixepsilon T_{\overline{CP}}
\end{equation}
where $\varepsilon\equiv 4\sum_{i=1}^m \varepsilon_{1,\dots,i}\to 0$ as $\varepsilon_i\to 0$, $i=1,\dots,m$.
\end{lemma}
\begin{proof}
Let $x\in\Fix T_{CDR\lambda}\cap U$.
Set $x_{m+1}=x$. Set $x_i=T_{i,i+1}x_{i+1}$, for all $i=m,\dots,1$. Using \eqref{eq:x1.equal.x}, we get $x_1=x$.
We claim that
\begin{equation}\label{eq:loop.proj}
\forall i=1,\dots,m\,,\, \exists v_i\in \mathbb B\,:\,P_{A_1}P_{A_2}\dots P_{A_i}x_i=P_{A_1}P_{A_2}\dots P_{A_i}P_{A_{i+1}}x_{i+1}+4\varepsilon_{1,\dots,i} v_i.
\end{equation}
Indeed, let $i\in\{1,\dots,m\}$.
By Lemma \ref{lem:Tij.expression}, we have
\begin{equation}\label{e:Gibbons}
x_i=\lambda P_{A_i}R_{A_{i+1}}x_{i+1}-\lambda R_{A_{i+1}}x_{i+1}+P_{A_{i+1}}x_{i+1}
\end{equation}
(see \eqref{eq:rep.x}).
Since $x_i=T_{i,\dots,m}x$, by \eqref{eq:subset.Ui}, this implies that
$x_i$, $P_{A_i}R_{A_{i+1}}x_{i+1}$, $R_{A_{i+1}}x_{i+1}$, and $P_{A_{i+1}}x_{i+1}$
all belong to $U_i$.
By Lemma \ref{lem:composition.affine.m}, $P_{A_1} \dots P_{A_i}$ is approximated by the affine operator $\Phi_{1,\dots,i}$ with error $\varepsilon_{1,\dots,i}$ on $U_i$.
Then there exists $u_j\in \mathbb B$, $j=1,\dots,4$, such that 
\begin{equation}
\begin{array}{rl}
&P_{A_1} \dots P_{A_i}x_i= \Phi_{1,\dots,i} x_i +\varepsilon_{1,\dots,i} u_1,\\
& P_{A_1} \dots P_{A_i} P_{A_i}R_{A_{i+1}}x_{i+1}=\Phi_{1,\dots,i} P_{A_i}R_{A_{i+1}}x_{i+1}+\varepsilon_{1,\dots,i} u_{2},\\
&P_{A_1} \dots P_{A_i} R_{A_{i+1}}x_{i+1}=\Phi_{1,\dots,i} R_{A_{i+1}}x_{i+1}+\varepsilon_{1,\dots,i} u_3,\\
&P_{A_1} \dots P_{A_i}P_{A_{i+1}}x_{i+1}=\Phi_{1,\dots,i}P_{A_{i+1}}x_{i+1}+\varepsilon_{1,\dots,i} u_4.
\end{array}
\end{equation}
This together with \eqref{e:Gibbons} yields
\begin{equation}
\begin{array}{rl}
P_{A_1} \dots P_{A_i}x_i=&\Phi_{1,\dots,i} x_i +\varepsilon_{1,\dots,i} u_1\\
=&\Phi_{1,\dots,i} (\lambda P_{A_1}R_{A_{i+1}}x_{i+1}-\lambda R_{A_{i+1}}x_{i+1}+P_{A_{i+1}}x_{i+1}) +\varepsilon_{1,\dots,i} u_1\\
=&\lambda \Phi_{1,\dots,i} P_{A_i}R_{A_{i+1}}x_{i+1}-\lambda \Phi_{1,\dots,i} R_{A_{i+1}}x_{i+1}\\
& \qquad+\,\Phi_{1,\dots,i}P_{A_{i+1}}x_{i+1} +\varepsilon_{1,\dots,i}  u_1\\
=&\lambda (P_{A_1} \dots P_{A_i} P_{A_i}R_{A_{i+1}}x_{i+1}-\varepsilon_{1,\dots,i} u_2)\\
& \qquad-\,\lambda (P_{A_1} \dots P_{A_i} R_{A_{i+1}}x_{i+1}-\varepsilon_{1,\dots,i} u_3)\\
&\qquad+\, (P_{A_1} \dots P_{A_i}P_{A_{i+1}}x_{i+1}-\varepsilon_{1,\dots,i} u_4) +\varepsilon_{1,\dots,i}  u_1\\
=&P_{A_1} \dots P_{A_i}P_{A_{i+1}}x_{i+1}+\varepsilon_{1,\dots,i}(-\lambda u_2+\lambda u_3-u_4+u_1)\\
=&P_{A_1} \dots P_{A_i}P_{A_{i+1}}x_{i+1} +4\varepsilon_{1,\dots,i} v_i,
\end{array}
\end{equation}
where $v_i=\frac{1}{4}(-\lambda u_2+\lambda u_3-u_4+u_1)$ with
\begin{equation}
\|v_i\|\le \frac{1}{4}(\lambda \|u_2\|+\lambda \|u_3\|+\|u_4\|+\|u_1\|)\le \frac{1}{4}(1+1+1+1)=1\,.
\end{equation}
Since $x=x_1$, by \eqref{eq:loop.proj}, we get
\begin{equation}
    \begin{array}{rl}
         P_{A_1}x=P_{A_1}x_1=& P_{A_1}P_{A_2}x_2+4\varepsilon_1 v_1\\ 
         =&P_{A_1}P_{A_2}P_{A_3}x_3+4\varepsilon_{1,2} v_2+4\varepsilon_1 v_1\\
         =& \dots\\
        =&P_{A_1}\dots P_{A_m}P_{A_{m+1}}x_{m+1}+4\varepsilon_{1,\dots,m} v_m+\dots+4\varepsilon_1 v_1\\
        =&T_{\overline{CP}}x +\varepsilon v = T_{\overline{CP}}P_{A_1}x + \varepsilon v,
    \end{array}
\end{equation}
where $v=\frac{4}{\varepsilon}\sum_{i=1}^m \varepsilon_{1,\dots,i} v_i$.
Note that 
\begin{equation}
\|v\|\le \frac{4}{\varepsilon}\sum_{i=1}^m \varepsilon_{1,\dots,i} \|v_i\|\le \frac{4}{\varepsilon} \sum_{i=1}^m \varepsilon_{1,\dots,i}=1.
\end{equation}
Thus $P_{A_1}x\in\Fixepsilon T_{\overline{CP}}$, as claimed.
\end{proof}

The following lemma provides sufficient conditions on a set allowing approximation of its projector by the
projection onto its translated tangent space.
\begin{lemma}\label{lem:suff.cond.approxi}
Let $A$ be a subset of $\Ebb$, and let $\bar x\in A$.
Assume that the following conditions hold:
\begin{enumerate}[(a)]
    \item $A$ is $\epsilon$-super-regular  at $\bar x$ on the neighborhood $W$ of $\bar x$ with $\epsilon_W>0$,
    that is $A$ satisfies \eqref{eq:ele.subreg} at $\bar x$ on $W$ with constant $\epsilon_W$.
    \item $A$ has Shapiro$^+$ property \eqref{e:Shapiro+} at $\bar x$ on a neighborhood $V$ of $\bar x$ with constant $k>0$.
\end{enumerate}
\medskip

\noindent Set $\hat{\epsilon}:= 4\epsilon_W/(1-\epsilon_W)^2.$ 
 Let $\bar\epsilon>0$ and define $U:=W\cap V\cap B(\bar x,r)$, where $r=\frac{\bar\epsilon}{\sqrt{1+\hat{\epsilon}}+k+1}$.
Then $P_A$ is approximated by the projection onto the translated tangent space, $P_{\mathcal{T}_{A}(\bar x)+\bar x}$,
with error $\bar\epsilon$ on $U$.
\end{lemma}
\begin{proof}
Note that $P_A \bar x=\bar x$.
By Lemma \ref{lem:nonexp.refector.projector} and condition (a), $P_A$ is pointwise almost
nonexpansive at $\bar x$ on  $W$ with violation $\hat{\epsilon}$, i.e.,
\begin{equation}\label{eq:nonexpansive.PC}
    \|x^+-\bar x\|\le \sqrt{1+\hat{\epsilon}}\|x-\bar x\|\,,\,\forall x\in W\,,\,x^+\in P_A x\,.
\end{equation}
By condition (b), we get
\begin{equation}\label{eq:Shapiro.prop}
    d_{\mathcal{T}_{A}(\bar x)}(x-\bar x)\leq k\|x-\bar x\|,\ \forall x\in V.
\end{equation}
Let $x\in U$. Let $x^+\in P_A x$ and $z\in P_{\mathcal{T}_{A}(\bar x)+\bar x}x$.
We have
\begin{equation}
    \begin{array}{rl}
         \|x^+-z\|=&\|(x^+-\bar x) +[(x-\bar x)+(\bar x- z)]-(x-\bar x)\|  \\[5pt]
         \leq&\|x^+-\bar x\|+\|(x-\bar x)+(\bar x- z)\|+\|x-\bar x\|\\[5pt]
         \leq & \sqrt{1+\hat{\epsilon}}\|x-\bar x\|+\|(x-\bar x)-(z-\bar x)\|+\|x-\bar x\|\\[5pt]
         =&\sqrt{1+\hat{\epsilon}}\|x-\bar x\|+d_{\mathcal{T}_{A}(\bar x)}(x-\bar x)+\|x-\bar x\|\\[5pt]
         \leq &\sqrt{1+\hat{\epsilon}}\|x-\bar x\|+k\|x-\bar x\|+\|x-\bar x\|\\[5pt]
         =&(\sqrt{1+\hat{\epsilon}}+k+1)\|x-\bar x\|\le (\sqrt{1+\hat{\epsilon}}+k+1)r=\bar\epsilon.
    \end{array}
\end{equation}
The second inequality follows from \eqref{eq:nonexpansive.PC}.
The second equality relies on $z-\bar x\in P_{\mathcal{T}_{A}(\bar x)+\bar x}x-\bar x= P_{\mathcal{T}_{A}(\bar x)}(x-\bar x)$.
The third inequality is a consequence of \eqref{eq:Shapiro.prop}.
Hence the result follows.
\end{proof}

\begin{proposition}[affine sets]\label{t:T_S line-Shapiro+}
If a set $S\subset\Ebb$ has a tangent cone $\mathcal T_S(\bar x)$ that is a subspace, then $S$ has the
Shapiro$^+$ property at $\bar x$ with constant
\begin{equation}\label{eq:Shapiro number}
 k=\|\mathcal A^\top (\mathcal A\mathcal A^\top)^{-1}\mathcal A\|,
\end{equation}
where $\mathcal A$ is the matrix characterizing $\mathcal T_S(\bar x)$.
\end{proposition}

\begin{proof}
Suppose the dimension of $\Ebb$ is $n$.
Since $S$ has the tangent cone $\mathcal{T}_{S}(\bar x)$ that is a linear subspace of $\Ebb$,
there exists a $m\times n$ matrix,  $\mathcal{A}$,  of rank $m\leq n$ such that
$\mathcal{T}_{S}(\bar x)=\{x\in\Ebb\,|\, \mathcal Ax=0\}$.
Let $z\in \mathbb E$.
Now, we apply the formula for the projector onto the subspace $\mathcal{T}_{S}(\bar x)$, namely
$P_{\mathcal{T}_{S}(\bar x)}z=z-\mathcal A^\top (\mathcal A\mathcal A^\top)^{-1}\mathcal Az$.
This yields
\[
d_{\mathcal T_{S}(\bar x)}(z) =\|z-P_{\mathcal T_{S}(\bar x)}z\|  =\|\mathcal A^\top (\mathcal A\mathcal A^\top)^{-1}\mathcal Az\|\le k\|z\|,
\]
where $k=\|\mathcal A^\top (\mathcal A\mathcal A^\top)^{-1}\mathcal A\|$.
Thus $S$ has Shapiro$^+$ property at $\bar x$ on $\Ebb$ with constant $k$.
\end{proof}

The following theorem says that if the sets are $\epsilon$-super-regular   and have tangent cones that are in
fact linear spaces, then
the shadows of fixed points of the cyclic relaxed Douglas-Rachford operator are almost fixed points
of the cyclic projection operator.
\begin{theorem}\label{theo:shadow.fixed.point.subset.CP.noise}
Let $A_j$ be nonempty closed subsets of $\mathbb E$, $j=1,\dots,m$.
Denote $A_{m+1}:=A_1$.
Let $\lambda\in[0,1]$.
Set $T_{i,\dots,m}=T_{i,i+1}\circ T_{i+1,i+2}\circ\dots\circ T_{m,m+1}$ where $T_{j,j+1}$ is defined as in \eqref{eq:DRl.Ai}.
Let $U$ be a subset of $\mathbb E$. 
Assume that the following conditions hold:
\begin{enumerate}[(a)]
    \item $\Fix T_{CDR\lambda}\cap U\neq\emptyset$ and  $T_{i,i+1}$ is single-valued on $T_{i+1,\dots,m}(\Fix T_{CDR\lambda}\cap U)$, $i=1,\dots,m$.
    \item For $i=1,\dots,m+1$, $A_i$ is $\epsilon$-super-regular  at $\bar x_i\in A_i$ on
    neighborhood $W_i$ of $\bar x_i$ with constant $\epsilon_{W_i}>0$.
    \item For $i=1,\dots,m+1$, the tangent cone to $A_i$ at $\bar x_i$,
    denoted by $\mathcal{T}_{A_i}(\bar x_i)$, is a linear subspace of $\Ebb$.\\~\\
Set $\hat{\epsilon}_i:= 4\epsilon_{W_i}/(1-\epsilon_{W_i})^2.$
For a fixed $\bar\epsilon_i>0$, define $U_i=W_i\cap V_i\cap B(\bar x_i,r_i)$,
where $r_i=\frac{\bar\epsilon_i}{\sqrt{1+\hat{\epsilon}_i}+k_i+1}$
where $k_i$ is the Shapiro$^+$ number  given by \eqref{eq:Shapiro number}.
    \item For $i=1,\dots,m-1$, $P_{A_{i+1}}U_{i+1}
    \subseteq U_i$, and for $i=1,\dots,m$,
\begin{equation}
\begin{array}{rl}
&T_{i,\dots,m}(\Fix T_{CDR\lambda}\cap U)\subseteq U_i,\\[5pt]
&P_{A_i}R_{A_{i+1}}T_{i+1,\dots,m}(\Fix T_{CDR\lambda}\cap U)\subseteq U_i,\\[5pt]
&R_{A_{i+1}}T_{i+1,\dots,m}(\Fix T_{CDR\lambda}\cap U)\subseteq U_i,\\[5pt]
&P_{A_{i+1}}T_{i+1,\dots,m}(\Fix T_{CDR\lambda}\cap U)\subseteq U_i\,.
\end{array}
\end{equation}
\end{enumerate}
For $i=1,\dots,m$, set $L_{i}=T_{A_{i}}(\bar x_{i})+\bar x_{i}$ and for $i=1,\dots,m-1$ define
$T_{\overline{CP}}=P_{A_1}\circ P_{A_2}\circ\dots\circ P_{A_{m}}\circ P_{A_1}$.

Then
\begin{equation}\label{eq:shadow.sunset.CP.noise.thm}
    P_{A_1}(\Fix T_{CDR\lambda}\cap U)\subseteq \Fix_{\!\!\bar\epsilon}\, T_{\overline{CP}}
\end{equation}
where
\begin{subequations}
\begin{equation}\label{e:varepsiloni+1}
\bar\epsilon=4\sum_{i=1}^m \bar\epsilon_{1,\dots,i}\quad \mbox{ for }\quad
\bar\epsilon_{1,\dots,i+1}:=\bar\epsilon_{1,\dots,i}+\bar\epsilon_{i+1} \sup_{w\in\mathbb B}{\|\Phi_{1,\dots,i}w\|}
\end{equation}
with
\begin{equation}\label{e:Phi_i}
\Phi_{1,\dots,i+1}(x):=(1-\bar\epsilon_{i+1})\Phi_{1,\dots,i}\left(\frac{1}{1-\bar\epsilon_{i+1}}\Phi_{i+1}(x)\right)\,,
\quad\mbox{for }\, \Phi_i(x):=P_{L_i}x
\end{equation}
($i=1,\dots,m$).
\end{subequations}
Moreover $\bar\epsilon\to 0$
as $\bar\epsilon_i\to 0$ for all $i$.
\end{theorem}
\begin{proof}
   By Lemma \ref{lem:suff.cond.approxi} and Proposition \ref{t:T_S line-Shapiro+},  conditions (b) and (c)
   imply that $P_{A_i}$ is approximated by $P_{L_i}$ with error $\bar\epsilon_i$ on $U_i$, $i=1,2,\dots,m$.
   Using Lemma \ref{lem:shadow.fixed.point.subset.CP.noise}, we obtain the result \eqref{eq:shadow.sunset.CP.noise.thm}.
\end{proof}
\begin{remark}
The heuristic of using one or more projections to clean up (approximate) fixed points of the CDR$\lambda$ mapping
can now be fully justified.
    For $\bar{\epsilon}$ small enough, $\Fix_{\!\!\bar{\epsilon}}T_{\overline{CP}}$ approximates $\Fix T_{\overline{CP}}$ and
    clearly
    \begin{equation}
    (T_{\overline{CP}})^n\Fix_{\!\!\bar{\epsilon}}T_{\overline{CP}}\underset{n}{\to} \Fix T_{\overline{CP}}
    \end{equation}
    where convergence is at the local rate of convergence of the cyclic projections mapping, and presumably not many
    iterations are required to achieve a desired tolerance.
Based on Theorem \ref{theo:shadow.fixed.point.subset.CP.noise}, we obtain
\begin{equation}\label{eq:hybrid.CDRl.CP}
T_{\overline{CP}}(\Fix T_{CDR\lambda}\cap U)=T_{\overline{CP}}(P_{A_1}(\Fix T_{CDR\lambda}\cap U))\subseteq T_{\overline{CP}}(\Fix_{\!\!\bar\epsilon}\, T_{\overline{CP}}).
\end{equation}
From this, it follows that 
\begin{equation}
(T_{\overline{CP}})^n(\Fix T_{CDR\lambda}\cap U) \subseteq (T_{\overline{CP}})^n(\Fix_{\!\!\bar\epsilon}\, T_{\overline{CP}})\underset{n}{\to} \Fix T_{\overline{CP}}=\Fix T_{CP},
\end{equation}
where $T_{CP}=P_{A_1}\dots P_{A_m}$.
The last equality follows from Remark \ref{rem:fix.CP.bar.eq.fix.CP}.
Again, for small enough neighborhoods, and with sets that are smooth enough, then only a few
iterations $n$ of cyclic projections
will be required to map fixed points of the cyclic relaxed Douglas-Rachford mapping to
neighborhoods of fixed points of cyclic projections, where
interpretation is easier.
\end{remark}

\if{The fixed point operator that we study, $T$ is built from compositions of relaxed Douglas-Rachford operators, and these are built from compositions and over-relaxations of projectors, hence by Lemma \ref{t:av-comp av} the the $\alpha$-firm nonexpensiveness of $T$  will come from the the $\alpha$-firm nonexpensiveness of the individual projectors.  The {\color{blue}the $\alpha$-firm nonexpensiveness} of the projectors, moreover, is inherited by the regularity of the sets, so at the bottom of all of this is set-regularity. }\fi

\section{Quantitative convergence analysis}\label{s:quant}

The results of this section rely on the following standing assumptions.
\begin{assumption}\label{a:1}  The following hold for all  $j=1,2\dots, m\geq 2$ with $m+1\equiv 1$.
\begin{enumerate}[(a)]
\item\label{a:1,a} The sets  $A_j$ are nonempty and closed subsets of $\Ebb$.
\item\label{a:1,b}  There is a neighborhood $U_j$ of each $x_j\in A_j$ and a set
$\Lambda_j:=P^{-1}_{A_j} (A_j\cap U_j)\cap U_j$ on which
$A_j$ is $\epsilon$-super-regular at $x_j$ relative to $\Lambda_j$ with constant $\epsilon_{U_j}$
    on the neighborhood $U_j$.
\item\label{a:1,c} On the same neighborhoods and sets defined in (b), the reflectors satisfy\\
$R_{A_{j+1}}U_{j+1}\subseteq U_{j}$ and $R_{A_{j+1}}\Lambda_{j+1}\subseteq \Lambda_{j}$.
\item\label{a:1,d}  $T_{j,j+1}\Lambda_{j+1}\subseteq \Lambda_{j}$ and
$T_{j,j+1}U_{j+1}\subseteq U_{j}$ for $T_{j,j+1}$ defined by \eqref{eq:DRl.Ai}.
\end{enumerate}
\end{assumption}

The next lemma establishes the constant and violation for relaxed Douglas--Rachford operators constructed from
pointwise a$\alpha$-fne mappings.  This result is very similar to \cite[Theorem 3.7]{luke2020convergence}, but with
more detailed values of the parameters.
\begin{proposition}\label{lem:Tfirmly.nonep}
Let $\lambda\in [0,1]$ be fixed and let Assumption \ref{a:1}\eqref{a:1,a}-\eqref{a:1,c} hold
for the collection of sets $\{A_1, A_2, \dots, A_m\}$.
The two-set relaxed Douglas-Rachford mapping defined by \eqref{eq:DRl.Ai}
is pointwise a$\alpha$-fne at all $y\in \Lambda_{j+1}$ with constant $\alpha_{j,j+1}=1/2$ and
violation $\epsilon_{j,j+1}$ on $U_{j+1}$
where
\begin{equation}
    \epsilon_{j,j+1}:=\frac{1}{2}\left[\left(\lambda\sqrt{1+\tilde\epsilon_{j}}+1-\lambda\right)^2(1+\tilde\epsilon_{j+1})-1\right]
\end{equation}
for
$\tilde{\epsilon}_j= 8\epsilon_{U_j}(1 + \epsilon_{U_j})/(1 -\epsilon_{U_j})^2$ \,$(j=1,2,\dots, m)$.

If in addition Assumption \ref{a:1}\eqref{a:1,d} holds, then the cyclic relaxed Douglas-Rachford mapping
\eqref{eq:TCDRl} is pointwise
a$\alpha$-fne at all $y\in \Lambda_1$ on $U_1$ with
\begin{equation}\label{eq:violation.composite}
 \alpha = \frac{m}{m+1}
\quad\mbox{ and violation }\quad
    \epsilon=\prod_{j=1}^m(1+\epsilon_{j,j+1})-1.
\end{equation}
\end{proposition}
\begin{proof}
Let $y\in \Lambda_{j+1}$ and  $x\in U_{j+1}$.
By Lemma \ref{lem:average char}, it is sufficient to prove that $\tilde{T}:= 2T_{j,j+1}-\Id$
is pointwise almost nonexpensive at $y$ with violation $2\epsilon_{j,j+1}$ on $U_{j+1}$.
Starting with \eqref{eq:DRl.Ai} we derive an equivalent representation for
$\tilde{T}$:
\begin{align}\label{e:TR/ST}
    \tilde{T} &=2\left(\frac{\lambda}{2}(R_{A_j} R_{A_{j+1}} +\Id)+(1-\lambda)P_{A_{j+1}}  \right)-\Id\\
&=\lambda (R_{A_j} R_{A_{j+1}} +\Id) +2(1-\lambda)P_{A_{j+1}} -\Id \\
&=\lambda R_{A_j} R_{A_{j+1}} + 2(1-\lambda)P_{A_{j+1}} + (\lambda-1)\Id\\
&=\lambda R_{A_j} R_{A_{j+1}} + (1-\lambda)(2P_{A_{j+1}}-\Id)\\
&=\lambda R_{A_j} R_{A_{j+1}} + (1-\lambda)R_{A_{j+1}}\\
&=\left(\lambda R_{A_j}+ (1-\lambda)\Id\right) R_{A_{j+1}}.
\end{align}
Let $y^+\in \tilde{T}y$ and $x^+\in \tilde{T}x$. We claim that
\begin{equation}
    \|y^+-x^+\|\leq\sqrt{1+2\epsilon_{j,j+1}}\|x-y\|.
\end{equation}
Indeed, let $u\in R_{A_{j+1}}y$, $z\in R_{A_{j+1}}x$, $u'\in R_{A_j}u$, $z'\in R_{A_j}z$ so
that by \eqref{e:TR/ST}
\begin{align*}
\tilde{T}y &= \left(\lambda R_{A_j}+ (1-\lambda)\Id\right) R_{A_{j+1}}y\ni
\lambda u' + (1-\lambda)u,\\
\tilde{T}x &= \left(\lambda R_{A_j}+ (1-\lambda)\Id\right) R_{A_{j+1}}x
\ni \lambda z' + (1-\lambda)z.
\end{align*}
It follows by Assumption \ref{a:1}\eqref{a:1,a} and \eqref{a:1,c}
that $u\in R_{A_{j+1}}\Lambda_{j+1}\subseteq \Lambda_j$
and $z\in R_{A_{j+1}}U_{j+1}\subseteq U_j$.
By Assumption \ref{a:1}\eqref{a:1,a}-\eqref{a:1,b} and Lemma \ref{lem:nonexp.refector.projector}, $R_{A_{j+1}}$
is pointwise almost nonexpansive with
violation $\tilde \epsilon_{j+1}$ given by \eqref{e:epstilde} with $\epsilon_U$ replaced by
$\epsilon_{U_{j+1}}$ at each point in $\Lambda_{j+1}$ on $U_{j+1}$.  $R_{A_{j+1}}$ is therefore also single-valued
at every point in $\Lambda_{j+1}$ so we can write $u = R_{A_{j+1}}y$,  and
\begin{subequations}
\begin{equation}\label{eq:nonexp.RA2}
    \|z-u\|\leq \sqrt{1+\tilde\epsilon_{j+1}}\|x-y\|.
\end{equation}
By the same argument, $R_{A_j}$ is pointwise almost nonexpansive (and hence single-valued)
with violation $\tilde \epsilon_j$
given by \eqref{e:epstilde} with $\epsilon_U$ replaced by
$\epsilon_{U_j}$ at all points in $\Lambda_j$ on
$U_j$. We can therefore write $u'=R_{A_j}u$, and
\begin{equation}\label{eq:nonexp.RA1}
    \|z'-u'\|\leq \sqrt{1+\tilde\epsilon_j}\|z-u\|.
\end{equation}
\end{subequations}
Let $u'' = \lambda u'+(1-\lambda)u $ and $z'' = \lambda z'+(1-\lambda)z$.
Then $u''\in\tilde{T}y$ and $z''\in \tilde{T}x$. We have
\begin{align*}
    \|z''-u''\| &=\|\lambda z'+(1-\lambda)z-\lambda u'-(1-\lambda)u  \|\\
    &\leq \lambda\|z'-u'\| +(1-\lambda)\|z-u\|\\
    &\leq \lambda \sqrt{1+\tilde\epsilon_1}\|z-u\| +(1-\lambda)\|z-u\|\\
    &= \left(\lambda\sqrt{1+\tilde\epsilon_j}+1-\lambda\right)\|z-u\| \\
    &\leq \left(\lambda\sqrt{1+\tilde\epsilon_j}+1-\lambda\right)\sqrt{1+\tilde\epsilon_{j+1}}\|x-y\|\\
    &= \sqrt{1+2\epsilon_{j,j+1}}\|x-y\|.
\end{align*}
The first inequality follows from the triangle inequality.
The other  inequalities are a consequence of \eqref{eq:nonexp.RA1} and \eqref{eq:nonexp.RA2}.
Since $y\in\Lambda_{j+1}$ and $x\in U_{j+1}$ are arbitrary, the first claim follows.

Assumption \ref{a:1}\eqref{a:1,d} allows application of Lemma \ref{t:av-comp av} to complete the proof.
\end{proof}

\if{We illustrate points $x,R_{A_2}x,R_{A_1}R_{A_2}x,\tilde{T}x$ and $y,R_{A_2}y,R_{A_1}R_{A_2}y,\tilde{T}y$ from the proof of Proposition \ref{lem:Tfirmly.nonep} with $\lambda=\frac{1}{3}$ in Figure \ref{fig:firmly.nonexpansiveness.DRl}.
\begin{figure}
\centering
\definecolor{ududff}{rgb}{0.30196078431372547,0.30196078431372547,1}
\begin{tikzpicture}[line cap=round,line join=round,>=triangle 45,x=1cm,y=1cm]
\clip(-0.46,-2.45) rectangle (10.78,3.11);
\draw [line width=1pt] (-0.74,0.13) circle (3.4766650687116814cm);
\draw [line width=1pt] (9.62,1.03) circle (3.778306498948968cm);
\draw (7.3,-1.17) node[anchor=north west] {$A_2$};
\draw (1.12,-1.59) node[anchor=north west] {$A_1$};
\draw (8.3,0.7) node[anchor=north west] {$x$};
\draw (4.22,-1.8) node[anchor=north west] {$x'$};
\draw (-0.1,-0.11) node[anchor=north west] {$x''$};
\draw [line width=1pt,dash pattern=on 1pt off 1pt] (4.418784875866326,-1.8678198548744787)-- (0.585300418769549,-0.38324324506681995);
\draw [line width=1pt,dash pattern=on 1pt off 1pt] (4.418784875866326,-1.8678198548744787)-- (8.22,0.25);
\draw (7.8,1.8) node[anchor=north west] {$y$};
\draw (3.96,2.8) node[anchor=north west] {$y'$};
\draw (0.1,1) node[anchor=north west] {$y''$};
\draw [line width=1pt,dash pattern=on 1pt off 1pt] (0.824308755921165,0.7207650611193552)-- (4.200606987189753,1.9958324181245988);
\draw [line width=1pt,dash pattern=on 1pt off 1pt] (4.200606987189753,1.9958324181245988)-- (7.6,1.39);
\draw (3.1,0.5) node[anchor=north west] {$x'''$};
\draw (3.14,1.7) node[anchor=north west] {$y'''$};
\draw [line width=1pt] (7.6,1.39)-- (8.22,0.25);
\draw [line width=1pt] (0.824308755921165,0.7207650611193552)-- (0.585300418769549,-0.38324324506681995);
\draw [line width=1pt] (4.200606987189753,1.9958324181245988)-- (4.418784875866326,-1.8678198548744787);
\draw [line width=1pt,dash pattern=on 1pt off 1pt] (7.6,1.39)-- (0.824308755921165,0.7207650611193552);
\draw [line width=1pt,dash pattern=on 1pt off 1pt] (8.22,0.25)-- (0.585300418769549,-0.38324324506681995);
\draw [line width=1pt] (3.0828725039474434,0.9438433740795702)-- (3.1302002791796997,-0.17216216337788004);
\begin{scriptsize}
\draw [fill=ududff] (8.22,0.25) circle (2pt);
\draw [fill=ududff] (4.418784875866326,-1.8678198548744787) circle (2pt);
\draw [fill=ududff] (0.585300418769549,-0.38324324506681995) circle (2pt);
\draw [fill=ududff] (7.6,1.39) circle (2pt);
\draw [fill=ududff] (4.200606987189753,1.9958324181245988) circle (2pt);
\draw [fill=ududff] (0.824308755921165,0.7207650611193552) circle (2pt);
\draw [fill=ududff] (3.0828725039474434,0.9438433740795702) circle (2pt);
\draw [fill=ududff] (3.1302002791796997,-0.17216216337788004) circle (2pt);
\end{scriptsize}
\end{tikzpicture}
\caption{Pointwise almost firm nonexpansiveness of relaxed Douglas--Rachford operators. Setting $\lambda=\frac{1}{3}$, $x':=R_{A_2}x,x'':=R_{A_1}x',x''':=\tilde{T}x:=\lambda x''+(1-\lambda)x'$, $y':=R_{A_2}y,y'':=R_{A_1}y',y''':=\tilde{T}y:=\lambda y''+(1-\lambda)y'$.}
\label{fig:firmly.nonexpansiveness.DRl}
\end{figure}
}\fi

The next lemma uses the theory of {\em metric subregularity} of the {\em transport discrepancy}
related to a general fixed point mapping $T$ as developed in \cite{berdellima2022alpha}.
We do not go into details here, but
present this more directly in terms of the mapping $T$.  Recall that a
function $\rho:[0,\infty) \to [0,\infty)$ is a \textit{gauge function} when it is
continuous, strictly increasing
with $\rho(0)=0$, and $\lim_{t\to \infty}\rho(t)=\infty$.  For our purposes, we
will define this implicitly as follows.
\begin{align}\label{eq:gauge}
& \rho\paren{\paren{\frac{(1+\epsilon)t^2-\paren{\theta_{\alpha,\epsilon}(t)}^2}{\tfrac{1-\alpha}{\alpha}}}^{1/2}}=
 t\\
 &\qquad\qquad\qquad\quad\iff\\
& \theta_{\alpha,\epsilon}(t) = \paren{(1+\epsilon)t^2 - \tfrac{1-\alpha}{\alpha}\paren{\rho^{-1}(t)}^2}^{1/2},
\end{align}
where
$\map{\theta_{\alpha,\epsilon}}{[0,\infty)}{[0,\infty)}$
with parameters $\alpha\in [0,1)$ and $\epsilon\geq 0$ satisfying
\begin{eqnarray}\label{eq:theta_tau_eps}
(a)~ \theta_{\alpha,\epsilon}(0)=0; \quad (b)~ 0<\theta_{\alpha,\epsilon}(t)<t ~\forall t\in(0,\tbar]
\mbox{ for some }\tbar>0.\end{eqnarray}
The function $\theta_{\alpha,\epsilon}$ defined in this way is a gauge function whenever
$\rho$ is.

In the next theorem the parameter $\epsilon$ is
exactly the violation in a$\alpha$-fne mappings and $\alpha$ is the ``averaging'' constant.
In preparation for the results that follow, we will require at least one of the
additional assumptions on $\theta$.
\begin{assumption}\label{a:msr convergence}
The gauge $\theta_{\alpha,\epsilon}$ satisfies \eqref{eq:theta_tau_eps} and
at least one of the following holds.
 \begin{enumerate}[(a)]
\item\label{t:msr convergence, necessary sublin} $\theta_{\alpha,\epsilon}$ satisfies
\begin{equation}\label{eq:theta to zero}
    \theta_{\alpha,\epsilon}^{(k)}(t)\to 0\mbox{ as }k\to\infty~\forall t\in(0,\tbar),
\end{equation}
 and the sequence $(x^{(k)})$ is Fej\'er monotone with respect to $\Fix T$, i.e.
\begin{equation}\label{eq:Fejer}
\dist\paren{x^{(k+1)},\, y}\leq \dist(x^{(k)}, y) \quad \forall k\in\Nbb,
\forall y\in \Fix T;
\end{equation}
\item\label{t:msr convergence, necessary lin+}
$\theta_{\alpha,\epsilon}$ satisfies
\begin{equation}\label{eq:theta summable}
    \sum_{j=1}^\infty\theta_{\alpha,\epsilon}^{(j)}(t)<\infty~\forall t\in(0,\tbar)
\end{equation}
where $\theta_{\alpha,\epsilon}^{(j)}$ denotes the $j$-times composition of $\theta_{\alpha,\epsilon}$.
\end{enumerate}
\end{assumption}

In the case when the gauge $\rho$ is {\em linear}, then
\[
\rho(t)=\kappa t\quad\iff\quad
\theta_{\alpha, \epsilon}(t)=\paren{(1+\epsilon)-\frac{1-\alpha}{\alpha\kappa^2}}^{1/2}t\quad
\paren{\kappa\geq \sqrt{\tfrac{1-\alpha}{\alpha(1+\epsilon)}}}.
\]
The conditions in \eqref{eq:theta_tau_eps} in this
case simplify to $\theta_{\alpha, \epsilon}(t)=\gamma t$ where
\begin{equation}\label{eq:theta linear}
 0< \gamma\equiv \sqrt{1+\epsilon-\frac{1-\alpha}{\alpha\kappa^2}}<1\quad\iff\quad
\sqrt{\tfrac{1-\alpha}{\alpha(1+\epsilon)}}\leq  \kappa\leq \sqrt{\tfrac{1-\alpha}{\alpha\epsilon}}.
\end{equation}
In this case  $\theta_{\alpha, \epsilon}(t)$ satisfies
Assumption \ref{a:msr convergence}\eqref{t:msr convergence, necessary lin+}.
The weaker Assumption \ref{a:msr convergence}\eqref{t:msr convergence, necessary sublin}
characterizes sublinear convergence.

The next result is a restatement of \cite[Theorem 7]{luke23} for the setting of multi-valued mappings
in Euclidean spaces. \cite[Theorem 7]{luke23} concerns random selections from collections of
single-valued mappings, but the extension to multi-valued mappings presents no difficulty here since
we are in the deterministic case,
and logic of the proof of \cite[Theorem 7]{luke23}, specialized to a single mapping, works here as well.
\begin{proposition}\label{t:metric.subreg.convergence} Let $T:\Lambda\rightrightarrows \Lambda$ for
$\Lambda\subseteq\Ebb$ with $\Fix T$ nonempty and closed.
Suppose that, for all neighborhoods $U$ of $\Fix T$ small enough, there is an  $\bar\epsilon$
and $\bar\alpha<1$ such that
\begin{enumerate}[(a)]
   \item $T$ is pointwise a$\alpha$-fne at all $y\in \Fix T\cap \Lambda$ with constant $\bar\alpha$ and
       violation $\bar\epsilon$ on  $U$, and
   \item for the gauge $\rho$
   \begin{equation}\label{e:msr}
    \dist(x,\Fix T\cap \Lambda)\leq \rho(\dist(x, Tx))\quad\forall x\in U\cap \Lambda.
   \end{equation}
\end{enumerate}
Then, for any $x^{(k)}\in U\cap\Lambda$ close enough to $\Fix T\cap \Lambda$, the iterates $x^{(k+1)}\in Tx^{(k)}$ satisfy
\begin{equation}%
\dist(x^{(k+1)}, \Fix T\cap \Lambda)
\leq \theta_{\alpha, \epsilon}\paren{\dist\paren{x^{(k)}, \Fix T\cap \Lambda}}
\end{equation}%
where $\theta_{\alpha, \epsilon}$ is given by \eqref{eq:gauge}.
If in addition $\theta_{\alpha, \epsilon}$ satisfies one of the
conditions in Assumption \ref{a:msr convergence} with parameter values
$\bar\alpha$ and $\bar\epsilon$, then for all $x^{(0)}$ close enough to
$\Fix T\cap \Lambda$, $x^{(k)}\to \bar x\in \Fix T\cap \Lambda$
with rate
 $O\paren{\theta_{\tau, \epsilon}^{(k)}(t_0)}$ in case
 Assumption \ref{a:msr convergence}\eqref{t:msr convergence, necessary sublin}
 ($t_0\equiv\dist(x^{(0)}, \Fix T\cap \Lambda)$),
 and with rate  $O(s_k(t_0))$ for
$s_k(t)\equiv
\sum_{j=k}^\infty \theta_{\tau, \epsilon}^{(j)}(t)$
in case Assumption \ref{a:msr convergence}\eqref{t:msr convergence, necessary lin+}.
\end{proposition}
\begin{corollary}\label{t:msr convergence lin}
In the setting of Proposition \ref{t:metric.subreg.convergence} suppose that
for some fixed $\bar\alpha$ and for each $\bar\epsilon>0$
there is a neighborhood  $U_{\bar\epsilon}$ such that $T$ is pointwise
a$\alpha$-fne at all $y\in\Fix T\cap \Lambda$ with constant $\bar\alpha$ and
violation $\bar\epsilon$ on $U_{\bar\epsilon}$.   If in addition
\eqref{e:msr} holds with a linear gauge, i.e., $\rho(t)=\bar\kappa t$ for
some $\bar\kappa>0$, then for any $x^{(0)}\in \Lambda$ close enough to $\Fix T$,
the iterates $x^{(k+1)}\in Tx^{(k)}$ converge R-linearly to a point in $\Fix T$ and
\begin{equation}
    \dist(x^{(k+1)},\Fix T\cap \Lambda)\leq \bar c\, \dist(x^{(k)},\Fix T\cap \Lambda)
\end{equation}
where
$\bar c:=\sqrt{1+\bar\epsilon-\tfrac{1-\bar\alpha}{\bar\kappa^2\bar\alpha}}<1$.
\end{corollary}
\begin{proof}
 Since $T$ is pointwise a$\alpha$-fne with arbitrarily small violation $\bar\epsilon$ on
 $U_{\bar\epsilon}$ and satisfies \eqref{e:msr} with some $\bar\kappa>0$, then Assumption
 \ref{a:msr convergence}\eqref{t:msr convergence, necessary lin+} holds with
 $\theta_{\alpha, \epsilon}$ satisfying \eqref{eq:theta linear} with $\kappa\geq\bar\kappa$.
\end{proof}

The point of Corollary \ref{t:msr convergence lin} is that
if \eqref{e:msr} holds with a linear gauge on a neighborhood $U$, then
it holds with the same gauge on all smaller neighborhoods.  Therefore, if
the neighborhood $U_{\bar\epsilon}$ can be constructed so that the
violation $\bar\epsilon$ is arbitrarily small on that neighborhood,
then on all neighborhoods small enough,
at least R-linear convergence of the fixed point iteration is guaranteed.
For applications involving collections of prox-regular (weakly convex), {\em subtransversal}
sets this yields generic
empirical certificates for convergence of the algorithm to a fixed point
\cite{luke2020convergence,DinJanLuk24}.

We are now able to state the main result of this section.
\begin{theorem}[cyclic relaxed Douglas-Rachford algorithm: local convergence]\label{theo:convergence.CDRl}
Let $\{A_1, A_2, \dots, A_m\}$ be a collection of subsets of $\mathbb{E}$.
For $\lambda\in [0,1]$ fixed, let $T$ denote the cyclic relaxed Douglas-Rachford mapping
defined in \eqref{eq:TCDRl}.
Suppose that the following conditions hold:
\begin{enumerate}[(a)]
    \item\label{theo:convergence.CDRl a} $\Fix T\neq\emptyset$;
     \item\label{theo:convergence.CDRl b} There is a neighborhood $U_1$ of each $y\in\Fix T$ such that $A_1\cap U_1\neq\emptyset$ and for all $j=1,2,\dots,m$ there are neighborhoods $U_j$ of points in the sets $A_j$ and sets $\Lambda_j:=P^{-1}_{A_j} (A_j\cap U_j)\cap U_j$ where Assumption \ref{a:1} holds;
\item\label{theo:convergence.CDRl c} For the gauge $\rho$
   \begin{equation}\label{e:msr drl}
    \dist(x,\Fix T\cap \Lambda_1)\leq \rho(\dist(x, Tx))\quad\forall x\in U_1\cap \Lambda_1.
   \end{equation}
\end{enumerate}
Then the following hold.
\begin{enumerate}[(i)]
\item\label{theo:convergence.CDRl ii} For any $x^{(k)}\in U_1\cap\Lambda_1$ close enough to $\Fix T\cap \Lambda_1$,
the iterates $x^{(k+1)}\in Tx^{(k)}$ satisfy
\begin{equation}%
\dist\paren{x^{(k+1)}, \Fix T\cap \Lambda_1}
\leq \theta_{\alpha, \epsilon}\paren{\dist\paren{x^{(k)}, \Fix T\cap \Lambda_1}}
\end{equation}%
where $\theta_{\alpha, \epsilon}$ is given by \eqref{eq:gauge}.

If in addition $\theta_{\alpha, \epsilon}$ satisfies one of the
conditions in Assumption \ref{a:msr convergence} with the parameter values
$\alpha$ and $\epsilon$ above, then for all $x^{(0)}$ close enough to $\Fix T$
in $\Lambda_1$,  $x^{(k)}\to \bar x\in \Fix T\cap \Lambda_1$
with rate
 $O\paren{\theta_{\tau, \epsilon}^{(k)}(t_0)}$
 in case
 Assumption \ref{a:msr convergence}\eqref{t:msr convergence, necessary sublin}
 ($t_0\equiv\dist(x^{(0)}, \Fix T\cap \Lambda_1)$),
 and with rate  $O(s_k(t_0))$ for
$s_k(t)\equiv
\sum_{j=k}^\infty \theta_{\tau, \epsilon}^{(j)}(t)$
in case Assumption \ref{a:msr convergence}\eqref{t:msr convergence, necessary lin+}.

\item\label{theo:convergence.CDRl iii} If for some fixed $\bar\alpha$ and for each $\bar\epsilon>0$
there is a neighborhood  $U_{1, \bar\epsilon}$ such that $T$ is pointwise
a$\alpha$-fne at all $y\in\Fix T\cap \Lambda_1$ with constant $\bar\alpha$ and
violation $\bar\epsilon$ on $U_{1,\bar\epsilon}$ and
\eqref{e:msr} holds with a linear gauge, i.e., $\rho(t)=\bar\kappa t$ for
some $\bar\kappa>0$, then for any $x^{(0)}\in \Lambda_1$ close enough to $\Fix T$,
the iterates $x^{(k+1)}\in Tx^{(k)}$ converge R-linearly to a point in $\Fix T$ and
\begin{equation}\label{eq:CDRl.rate}
    \dist(x^{(k+1)},\Fix T\cap \Lambda_1)\leq \bar c\, \dist(x^{(k)},\Fix T\cap \Lambda_1)
\end{equation}
where
$\bar c:=\sqrt{1+\bar\epsilon-\tfrac{1-\bar\alpha}{\bar\kappa^2\bar\alpha}}<1$.
\end{enumerate}
\end{theorem}
\if{ (Figure \ref{fig:main.theo} plots sets $A_j,U_j,\Lambda_j$ and points $x^{(j)}$ in Theorem \ref{theo:convergence.CDRl}.)
\definecolor{xdxdff}{rgb}{0.49019607843137253,0.49019607843137253,1}
\definecolor{ududff}{rgb}{0.30196078431372547,0.30196078431372547,1}

\begin{figure}
    \centering
\definecolor{xdxdff}{rgb}{0.49019607843137253,0.49019607843137253,1}
\begin{tikzpicture}[line cap=round,line join=round,>=triangle 45,x=1cm,y=1cm]
\clip(-2.64,-4.21) rectangle (7.32,3.35);
\draw [line width=1pt] (-3.58,0.59) circle (4.976183276367542cm);
\draw [line width=1pt] (8.34,3.99) circle (5.254217353707402cm);
\draw [line width=1pt] (4.64,-4.79) circle (3.4733269353747858cm);
\draw (0.3,2.33) node[anchor=north west] {$A_1$};
\draw (5.72,0.11) node[anchor=north west] {$A_2$};
\draw (5.64,-1.7) node[anchor=north west] {$A_3$};
\draw [line width=1pt] (1.2281621474857234,-0.6921765726628596) circle (1.7cm);
\draw [line width=1pt] (3.901136952479928,1.178720069903954) circle (2cm);
\draw [line width=1pt] (3.350038760272231,-1.56509690068058) circle (1.5cm);
\draw (1.06,-0.7) node[anchor=north west] {$x^{(1)}$};
\draw (3,-1.5) node[anchor=north west] {$x^{(3)}$};
\draw (3.3,1) node[anchor=north west] {$x^{(2)}$};
\draw (-0.4,0.1) node[anchor=north west] {$\Lambda_1=U_1$};
\draw (4,2) node[anchor=north west] {$\Lambda_2=U_2$};
\draw (2.7,-2) node[anchor=north west] {$\Lambda_3=U_3$};
\begin{scriptsize}
\draw [fill=xdxdff] (1.2281621474857234,-0.6921765726628596) circle (2.5pt);
\draw [fill=xdxdff] (3.901136952479928,1.178720069903954) circle (2.5pt);
\draw [fill=xdxdff] (3.350038760272231,-1.56509690068058) circle (2.5pt);
\end{scriptsize}
\end{tikzpicture}
    \caption{Illustration of Theorem \ref{theo:convergence.CDRl} with $E=\R^2$, $m=3$,
    $A_j$s being circles, and $U_j$s being open balls centered at $x^{(j)}$s.
    }
    \label{fig:main.theo}
\end{figure}
}\fi
\begin{proof}
Assumption  (a) of Proposition \ref{t:metric.subreg.convergence} follows from
Assumptions \eqref{theo:convergence.CDRl a}-\eqref{theo:convergence.CDRl b}  and Proposition
\ref{lem:Tfirmly.nonep} with constants $\alpha$ and $\epsilon$ given by \eqref{eq:violation.composite}.
Assumption \eqref{theo:convergence.CDRl c} is assumption (b) of Proposition \ref{t:metric.subreg.convergence}.
The result then follows from Proposition \ref{t:metric.subreg.convergence} and Corollary
\ref{t:msr convergence lin}.
\end{proof}

\begin{remark}
Condition (c) in Theorem \ref{theo:convergence.CDRl} is difficult to check.
It was shown in \cite[Theorem 2]{LukTebTha18} that
linear metric subregularity is necesssary for local linear convergence.
In the applications to which we have applied the algorithm \cite{LukSabTeb19,DinJanLuk24} we have
not observed numerical behavior that is inconsistent with the hypothesis that
condition (c) in Theorem \ref{theo:convergence.CDRl} holds with a {\em linear} gauge.
For these applications the analysis of \cite{luke2020convergence}, which applies only
to relaxed Douglas-Rachford mappings, could also be adapted to provide generic
guarantees that, in particular for the application in \cite{DinJanLuk24},
condition (c) in Theorem \ref{theo:convergence.CDRl} is satisfied with a {\em linear} gauge.
\end{remark}

\begin{remark}[the convex case]
If the sets $A_1, A_2, \dots, A_m$ are all closed and convex, then global convergence statements are readily obtained.
In this case, Assumption \ref{a:1} is satisfied with $\Lambda_j=U_j=\Ebb$ and violations $\epsilon_{U_j} = 0$
for all $j=1,2,\dots,m$.  From Proposition \ref{lem:Tfirmly.nonep} it follows that the cyclic relaxed Douglas-Rachford
mapping \eqref{eq:TCDRl} is $\alpha$-fne with constant $\alpha=m/(m+1)$.  It follows then from classical arguments
that whenever there
exist fixed points, the sequence generated by the cyclic relaxed Douglas-Rachford
mapping \eqref{eq:TCDRl} converges to a fixed point, and this is independent of whether or not the
corresponding feasibility problem is consistent.  This improves the results of \cite{BorTam14,luke2018relaxed} which
require consistency of the feasibility problem.
\end{remark}

The next corollary analyzes the convergence of the sum of the distance between the sets $A_i$ at
any given iterate of the cyclic relaxed Douglas-Rachford algorithm:

\begin{corollary}\label{coro:conergence.gap}
In the setting of Theorem \ref{theo:convergence.CDRl}, assume in addition that\\
$P_{A_{i+1}}\Lambda_{i+1}\subseteq \Lambda_{i}$, and $P_{A_{i+1}}U_{i+1}\subseteq U_{i}$ for $i=m,\dots,1$.
Define
   \begin{equation}\label{eq:sum.gap}
    \mbox{gap}^{(k)}:=\sum_{i=1}^{m} \mbox{gap}^{((i,i+1),k)}, \ \forall k=1 ,2,\dots
\end{equation}
where $\mbox{gap}^{((i,i+1),k)}:= \|y^{(i+1,k)}-y^{(i,k)}\|$, for $i=m,\dots,1$.
Here $y^{(m+1,k)}\in P_{A_{m+1}} y^{(k)}$ and $y^{(i,k)}\in P_{A_i} y^{(i+1,k)}$, for $i=m,\dots,1$.
Then the following statements hold:
\begin{enumerate} 
    \item  
    The sequence $(\mbox{gap}^{(k)})_{k=0}^\infty$ converges to a constant $g\ge 0$.
    \item If $g=0$, problem \eqref{eq:multiset.model} is consistent.
\end{enumerate}
\end{corollary}
Figure \ref{fig:sum.gap} plots points $y^{(i,k)}$ and $\mbox{gap}^{((i,i+1),k)}$ in Corollary \ref{coro:conergence.gap}.
\begin{figure}
    \centering
   \definecolor{uuuuuu}{rgb}{0.26666666666666666,0.26666666666666666,0.26666666666666666}
\definecolor{ffqqqq}{rgb}{1,0,0}
\definecolor{ududff}{rgb}{0.30196078431372547,0.30196078431372547,1}
\begin{tikzpicture}[line cap=round,line join=round,>=triangle 45,x=1cm,y=1cm,scale=1.5]
\clip(-3.84,-0.330370370370369) rectangle (1.5288888888888883,3.0296296296296275);
\draw [line width=1pt,color=ffqqqq] (-4.62,2.41) circle (1.8539687160251654cm);
\draw [line width=1pt,color=ffqqqq] (2.06,2.15) circle (2.0932271735289505cm);
\draw [line width=1pt,color=ffqqqq] (-2.14,-0.83) circle (1.6141871019184857cm);
\draw [line width=1pt,dotted] (-1.58,1.55)-- (-2.8360420792455647,1.905327693470785);
\draw [line width=1pt,dash pattern=on 1pt off 1pt] (-2.8360420792455647,1.905327693470785)-- (-0.03061829890924768,2.045524508779318);
\draw [line width=1pt,dash pattern=on 1pt off 1pt] (-0.03061829890924768,2.045524508779318)-- (-1.1852330509832059,0.4715452635483573);
\draw [line width=1pt,dash pattern=on 1pt off 1pt] (-3.0054130136537665,1.4987871469731158)-- (-1.1852330509832059,0.4715452635483573);
\draw (-1.6533333333333329,1.6251851851851844) node[anchor=north west] {$y^{(k)}$};
\draw (-3.9007407407407396,2.8274074074074055) node[anchor=north west] {$A_1=A_4$};
\draw (-3.6,2.1288888888888877) node[anchor=north west] {$y^{(4,k)}$};
\draw (0,2.4311111111111097) node[anchor=north west] {$y^{(3,k)}$};
\draw (-1.6925925925925922,0.4407407407407412) node[anchor=north west] {$y^{(2,k)}$};
\draw (-3.8303703703703693,1.654814814814814) node[anchor=north west] {$y^{(1,k)}$};
\draw (0.45740740740740735,1.0074074074074) node[anchor=north west] {$A_3$};
\draw (-3.3229629629629622,0.4896296296296298) node[anchor=north west] {$A_2$};
\draw (-1.5703703703703698,2.512592592592591) node[anchor=north west] {$gap^{(3,4)}$};
\draw (-0.7170370370370368,1.239259259259259) node[anchor=north west] {$gap^{(2,3)}$};
\draw (-3.17037037037036,1.175555555555555) node[anchor=north west] {$gap^{(1,2)}$};
\begin{scriptsize}
\draw [fill=ududff] (-4.62,2.41) circle (1.5pt);
\draw [fill=ududff] (2.06,2.15) circle (1.5pt);
\draw [fill=ududff] (3.56,0.69) circle (1.5pt);
\draw [fill=ududff] (-2.14,-0.83) circle (1.5pt);
\draw [fill=ududff] (-1.24,-2.17) circle (1.5pt);
\draw [fill=ududff] (-7,2.49) circle (1.5pt);
\draw [fill=ududff] (-1.58,1.55) circle (1.5pt);
\draw [fill=uuuuuu] (-2.8360420792455647,1.905327693470785) circle (1pt);
\draw [fill=uuuuuu] (-0.03061829890924768,2.045524508779318) circle (1pt);
\draw [fill=uuuuuu] (-1.1852330509832059,0.4715452635483573) circle (1pt);
\draw [fill=uuuuuu] (-3.0054130136537665,1.4987871469731158) circle (1pt);
\draw [fill=ududff] (-4.573333333333333,1.4066666666666663) circle (1.5pt);
\end{scriptsize}
\end{tikzpicture}
    \caption{The sum of the gaps between the sets $A_i$. Here $\Ebb=\R^2$, $A_i$ are circles.}
    \label{fig:sum.gap}
\end{figure}
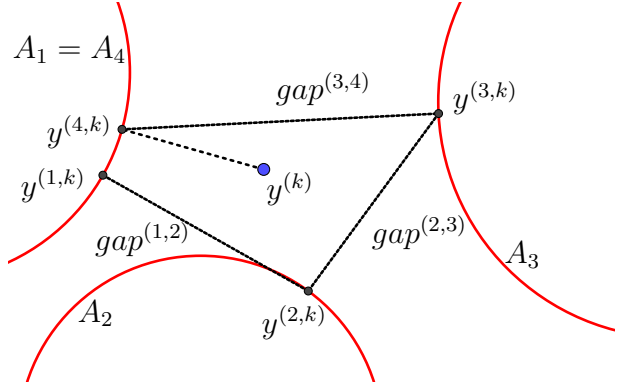
\begin{proof}
Let us prove the first statement.
Let $y^*\in\Fix T\cap\Lambda_{m+1}$ be limit point of $(y^{(k)})_{k=1}^\infty$.
Let $y^{(m+1,*)}\in P_{A_{m+1}} y^*$ and $y^{(i,*)}\in P_{A_i} y^{(i+1,*)}$, for $i=m,\dots,1$.
We claim that the sequence of shadow $(y^{(i,k)})_{k=0}^\infty$ converges to $y^{(i,*)}$ for $i=1,\dots,m+1$.
Indeed, by Assumption \ref{a:1},
Lemma \ref{lem:nonexp.refector.projector} implies that $P_{A_i}$ is pointwise almost
nonexpansive with violation $\hat \epsilon_i= 4\epsilon_i/(1-\epsilon_i)^2$
at all points
$u\in \Lambda_i$ on $U_i$, i.e., with $u'\in P_{A_i}u$,
\begin{equation}\label{eq:nonexp.RAi}
    \|v'-u'\|\leq \sqrt{1+\hat\epsilon_i}\|v-u\|, \ \forall v\in U_i\,,\, v'\in P_{A_i}v.
\end{equation}
Since $y^{(k)}\in\Lambda_{m+1}$, by assumption, we have $y^{(m+1,k)}\in P_{A_{m+1}}y^{(k)}\subseteq P_{A_{m+1}}\Lambda_{m+1}\subseteq\Lambda_m$. This implies 
\begin{equation}
\begin{array}{rl}
     y^{(m,k)}&\in P_{A_m}y^{(m+1,k)} \subseteq P_{A_m}\Lambda_{m}\subseteq  \Lambda_{m-1}\, ,\\
     y^{(m-1,k)}&\in  P_{A_{m-1}}y^{(m,k)} \subseteq P_{A_{m-1}}\Lambda_{m-1}\subseteq \Lambda_{m-2}\, ,\\
     \dots \\
     y^{(2,k)}&\in  P_{A_2}y^{(3,k)} \subseteq P_{A_2}\Lambda_2\subseteq \Lambda_{1}.
\end{array}
\end{equation}
Similarly, since $y^*\in\Lambda_{m+1}\subseteq U_{m+1}$, by assumption, we have $y^{(m+1,*)}\in P_{A_{m+1}}y^*\subseteq P_{A_{m+1}}U_{m+1}\subseteq U_m$. This implies
\begin{equation}
\begin{array}{rl}
     y^{(m,*)}&\in P_{A_m}y^{(m+1,*)} \subseteq P_{A_m}U_{m}\subseteq  U_{m-1}, \\
     y^{(m-1,*)}&\in  P_{A_{m-1}}y^{(m,*)} \subseteq P_{A_{m-1}}U_{m-1}\subseteq U_{m-2},\\
     \dots \\
     y^{(2,*)}&\in  P_{A_2}y^{(3,*)} \subseteq P_{A_2}U_2\subseteq U_{1}.
\end{array}
\end{equation}
By applying \eqref{eq:nonexp.RAi} for each $i=1,\dots,m+1$, we obtain
\begin{equation}
    \begin{array}{rl}
        \|y^{(1,k)}-y^{(1,*)}\| \le &  \sqrt{1+\hat\epsilon_{1}} \|y^{(2,k)}-y^{(2,*)}\|\, \\
          \le &\sqrt{(1+\hat\epsilon_{1})(1+\hat\epsilon_{2})} \|y^{(3,k)}-y^{(3,*)}\| \\
          &\dots\\
          \le&\sqrt{(1+\hat\epsilon_1)\dots(1+\hat\epsilon_{m})}\|y^{(m+1,k)}-y^{(m+1,*)}\| \\
          \le&\sqrt{(1+\hat\epsilon_1)\dots(1+\hat\epsilon_{m})(1+\hat\epsilon_{m+1})}\|y^{(k)}-y^{*}\|
    \end{array}
\end{equation}
Since $\|y^{(k)}-y^*\|\to 0$ as $k\to\infty$, the sequence $(y^{(i,k)})_{k=1}^\infty$ converges to $y^{(i,*)}$ for all $i=1,\dots,m+1$.
It implies that $(\mbox{gap}^{((i,i+1),k)})_{k=0}^\infty$ converges to  $\|y^{(i+1,*)}-y^{(i,*)}\|$ as $k\to\infty$, where $\mbox{gap}^{((i,i+1),k)}=\|y^{(i+1,k)}-y^{(i,k)}\|$ for all $i=m,\dots,1$.
Indeed, by the triangle inequality, we have
\begin{equation*}
    \begin{array}{rl}
 &\left|\mbox{gap}^{((i,i+1),k)}-\|y^{(i+1,*)}-y^{(i,*)}\|\right|\\[7pt]
 &= \left|\|y^{(i+1,k)}-y^{(i,k)}\|-\|y^{(i+1,*)}-y^{(i,*)}\|\right|\\[5pt]
 &=\left|\left\|\left[\left(y^{(i+1,k)}-y^{(i+1,*)}\right)-\left(y^{(i,k)}-y^{(i,*)}\right)\right] +\left(y^{(i+1,*)}-y^{(i,*)}\right)\right\|\right.\\
 &\qquad\qquad\qquad\qquad\qquad\qquad\qquad\qquad\qquad\quad \left.-\left\|y^{(i+1,*)}-y^{(i,*)}\right\|\right|\\
 &\le  \|(y^{(i+1,k)}-y^{(i+1,*)})-(y^{(i,k)}-y^{(i,*)})\|\\
 &\le  \|y^{(i+1,k)}-y^{(i+1,*)}\| +\|y^{(i,k)}-y^{(i,*)}\|.
    \end{array}
\end{equation*}
The right hand side converges to $0$ as  $k\to\infty$ so together with \eqref{eq:sum.gap},
the first statement follows with
\begin{equation}
    g=\sum_{i=1}^{m} \|y^{(i+1,*)}-y^{(i,*)}\|\ge 0\,.
\end{equation}
If $g=0$, then $y^{(1,*)}=\dots=y^{(m+1,*)}$.
Since $y^{(i,*)}\in P_{A_i} y^{(i+1,*)} \subseteq A_i$, we obtain $y^{(1,*)}=\dots=y^{(m+1,*)}\in  \bigcap_{i=1}^m A_i$.
Here $y^{(m+1,*)}$ is arbitrary in $P_{A_{m+1}}y^*$, which yields the second statement.
\end{proof}
\begin{remark}\label{rm:use.of.gap}
In practice, the limit $g$ of the $\mbox{gap}^{(k)}$ in equation \eqref{eq:sum.gap} provides valuable information regarding the proximity of the shadows of the iterate $y^{(k)}$ between the sets $A_i$.
If the problem \eqref{eq:multiset.model} is inconsistent, meaning that $\bigcap_{i=1}^m A_i=\emptyset$, then $g> 0$. 
In such cases, a smaller value of $g$ indicates a better solution in terms of finding the nearest points among the sets.
Thus $\mbox{gap}^{(k)}$ allows us to assess the quality of the approximate solution $y^{(k)}$.
\end{remark}

%
\bibliographystyle{abbrv}

\end{document}